\documentclass[10pt]{amsart}

\usepackage{amsthm,amsmath,amssymb,bm,comment}
\usepackage[utf8]{inputenc} 
\usepackage[english]{babel}
\usepackage[T1]{fontenc}
\usepackage{indentfirst}

\usepackage[shortlabels]{enumitem}
\usepackage{graphicx}
\usepackage[dvipsnames]{xcolor}
\usepackage[pagebackref,breaklinks,colorlinks=true,linkcolor=MidnightBlue,citecolor=MidnightBlue]{hyperref}

\usepackage{algorithm}
\usepackage{algpseudocode}

\newtheorem{thm}{Theorem}[section]
\newtheorem{lemma}[thm]{Lemma}
\newtheorem{corollary}[thm]{Corollary}
\newtheorem{prop}[thm]{Proposition}

\newtheorem{remark}[thm]{Remark}

\numberwithin{equation}{section}

\newcommand{\Rd}{\mathbb{R}^d}

\newcommand{\mt}{\mu\tau}

\newcommand{\id}{{\bf id}}

\newcommand{\norm}[1]{\lVert #1 \rVert} 

\newcommand{\Bignorm}[1]{\bigg\lVert #1 \bigg\rVert}

\newcommand{\RR}{\mathbb{R}}

\newcommand{\PP}{\mathbb{P}}

\DeclareMathOperator*{\argmin}{argmin}

\DeclareMathOperator{\cof}{cof}

\DeclareSymbolFont{bbold}{U}{bbold}{m}{n}
\DeclareSymbolFontAlphabet{\mathbbold}{bbold}

\newcommand{\cL}{\mathcal{L}}

\newcommand{\vp}{\varphi}
\newcommand{\tr}{tr}

\DeclareMathOperator{\sgn}{sgn}

\usepackage{amssymb}

\begin{document}
\title{Well-posedness and regularity for a Polyconvex energy}
\author{Wilfrid Gangbo, Matt Jacobs and Inwon Kim}
\maketitle

\begin{abstract}
We prove the existence, uniqueness, and regularity of minimizers of a polyconvex functional in two and three dimensions, which corresponds to the $H^1$-projection of measure-preserving maps. 
Our result introduces a new criteria on the uniqueness of the minimizer, based on the smallness of the lagrange multiplier. No estimate on the second derivatives of the pressure is needed to get a unique global minimizer. As an application, we construct a minimizing movement scheme to construct $L^r$-solutions of the Navier-Stokes equation (NSE) for a short time interval. Our scheme is an improved version of the split scheme introduced in Ebin--Marsden \cite{EM}, and allows us to solve the equation with $L^r$ initial data $(r>d)$ as opposed to $H^{d/2+5}$ initial data requirement in \cite{EM}. 
\end{abstract}

\section{Introduction} 


A problem of interest in nonlinear elasticity theory is the existence, uniqueness, and the characterization of the minimizers for variational problems of the form
\begin{equation}\label{eq:intro_functional}
 \inf_{ Z \in \mathcal{U} } \Bigl\{ \int_{\Omega}\Bigl( j(x, D Z) -F \cdot  Z \Bigr)dx\Bigr\}.
\end{equation}
Here, $\Omega\subset \RR^d$ is a bounded and open set with smooth boundary,  $j:\RR^d\times \RR^{d\times d}\to\RR\cup\{+\infty\}$ is a lower semicontinuous function, $\mathcal{U}$ is some appropriately chosen function space, and $F:\Omega\to\RR^d$ is a fixed function.  In the context of nonlinear elasticity theory, $\Omega$ represents a reference configuration occupied by an elastic body, $j$ represents the so-called stored energy density of the material, $F$ is an applied force, and $ Z$ represents the deformation undergone by the elastic body. 

\medskip

If we further assume that $j$ is a polyconvex functional (cf. \cite{Ball}), then the theory of the calculus of variations developed by Morrey \cite{Morrey} gives robust results on the existence of minimizers to problem (\ref{eq:intro_functional}). On the other hand, the uniqueness and regularity properties of minimizers or their characterizations in terms of a system of PDEs, remain a major challenge when $j$ fails to be convex (cf. e.g. \cite{Ball-lI} for a summary of a list of unsolved problems). We refer the reader to \cite{AcerbiF87} \cite{EvansLC87} \cite{EvansKneuss} for partial regularity results (up to a set of small measure) for a class of so-called quasiconvex stored energy density functions.

\medskip

An important example of polyconvexity arises in the study of incompressible materials.  Incompressibility can be encoded by requiring that the admissible set of deformation maps satisfies the determinant constraint $\det(DZ)=1$ everywhere.  If we let $\mathcal D{\rm iff}_{\id}(\Omega)$ denote the set of maps $X: \bar \Omega \rightarrow \bar \Omega$ that are volume preserving $C^1$--diffeomorphisms such that $X|_{\partial\Omega}=\id$, then the constraint functional
\[
\chi(Z)=\begin{cases}
0 & \textrm{if}\quad Z\in \mathcal D{\rm iff}_{\id}(\Omega); \\
+\infty &\textrm{otherwise},
\end{cases}
\]
is not convex with respect to $Z$, but is a convex function of $\det(DZ)$ (hence, $\chi$ is polyconvex).  

 In this paper, we are interested in studying a particular incompressible polyconvex functional.  Given an $H^1(\Omega)$ map $S:\Omega\to\Omega$, the so-called \emph{$H^1$-projection problem} seeks the closest incompressible map to $S$ in a weighted $H^1$ norm.  More explicitly, given a parameter $a>0$, one wishes to find a minimizer of the problem
\begin{equation}\label{eq:h1_proj}
J_a(Z):=\chi(Z)+ \frac{1}{2}\norm{Z-S}_{L^2(\Omega)}^2+\frac{a}{2}\norm{DZ-DS}_{L^2(\Omega)}^2.
\end{equation}
Note that since we are forced to choose $Z\in D{\rm iff}_{\id}(\Omega)$, the term $\frac{1}{2}\norm{Z-S}_{L^2(\Omega)}^2$ is equivalent to $-Z\cdot S$, thus the $H^1$ projection problem can be put into the form of problem (\ref{eq:intro_functional}).
 The intriguing problem \eqref{eq:h1_proj} dates back to the work of Lord Kelvin and continues to generate a lot of interest (cf. e.g. \cite{Benjamin, Brenier-arxiv14, Burton} and the discussion on the Navier-Stokes equations below).    The difficulty of problem (\ref{eq:h1_proj}) lies in the non-convexity of the constraint set $\mathcal D{\rm iff}_{\id}(\Omega)$.  Worse yet, in dimension $d=3$, the $H^1$ coercivity of the functional is not sufficient to deduce weak convergence of the determinant.  Thus, although (\ref{eq:h1_proj}) is a polyconvex problem, Morrey's theory cannot even ensure the \emph{existence} of a minimizer.   
 
 In this paper, we will develop a theory that allows us to deduce the existence and uniqueness of minimizers to the $H^1$-projection problem under a mild regularity assumption on the data $S$. 
Our main result in this paper can be summarized as follows:

\medskip

{\it In dimensions $d\in \{2,3\}$, if $r>d$ and $S$ is sufficiently close to the identity in $W^{2,r}(\Omega)$, then there exists a unique minimizer $Z^*\in W^{2,r}_{\id}(\Omega)$ of the $H^1$-projection problem (\ref{eq:h1_proj}).  }

\medskip

To give a more detailed explanation of our result, let us introduce the Euler-Lagrange equation associated to (\ref{eq:h1_proj}).  Formally, the tangent space to a point $Z\in \mathcal{D}{\rm iff}_{\id}(\Omega)$ is the space $\{v\circ Z: v\in C^1_0(\Omega), \, \nabla \cdot v=0\}$.   Therefore, any critical point of (\ref{eq:h1_proj}) must satisfy the equation
\[
(Z-S)\cdot v(Z)+a(DZ-DS):D(v(Z))=0,
\]
for every smooth divergence free vector field $v$ vanishing on the boundary.  Equivalently, every critical point $Z$ must have a corresponding scalar function $p:\Omega\to\RR$ such that
\begin{equation}\label{eq:crit}
(I-a\Delta)(Z-S)+\nabla p(Z)=0,
\end{equation}
where the equation should be interpreted distributionally.
As it turns out, one can also understand $p$ as a Lagrange multiplier for the determinant constraint.  If we define the Lagrangian
\begin{equation}\label{eq:lagrangian}
\mathcal{L}(Z,q)=\frac{1}{2}\norm{Z-S}_{L^2(\Omega)}^2+\frac{a}{2}\norm{DZ-DS}_{L^2(\Omega)}^2+\int_{\Omega} q(x)\big(|\det(DZ(x))|-1\big)\, dx,
\end{equation}
then the equations $\delta_Z \mathcal{L}(Z,q)=0$ and $\delta_{q}\mathcal{L}(Z,q)=0$ correspond to
\begin{equation}\label{eq:qcrit}
(I-a\Delta)(Z-S)+D^T\big(\cof(DZ) q\big)=0, \quad \det(DZ)=1,
\end{equation}
respectively.
Thanks to the Null-Lagrangian identity $D^T\cof(DZ)=0$, the equations (\ref{eq:crit}) and (\ref{eq:qcrit}) can be transformed into one another through the relation $q=p(Z)$. 

To prove our main theorem, we develop a new regularity condition on critical points $(Z^*, q^*)$ (i.e. solutions to equation (\ref{eq:qcrit})) that guarantee that $Z^*$ is the unique solution to the $H^1$-projection problem. When $q^*$ is bounded in $L^{\infty}$ and the singular values of $DZ^*$ are uniformly  bounded away from zero, we are able to show that the Lagrangian $\mathcal{L}(Z, q^*)$ has a previously undiscovered convexity property (c.f. Lemma \ref{lem:matrix_inequality} and Proposition \ref{pro:sufficient-cond}).  This property allows us to conclude that $Z^*$ is the unique minimizer of the relaxed problem $Z\mapsto  \mathcal{L}(Z ,q^*)$, and hence the original problem $Z\mapsto \sup_{q} \mathcal{L}(Z, q)$.  Let us note that although the focus is on $H^1$ projection problem in this work, our strategy can be generalized to other polyconvex problems in both compressible and incompressible non-linear elasticity. 
 
  Of course, our regularity condition is only useful if there actually exist critical points with the required properties. In order to find such points, we employ Ekeland's variational principle (EVP) \cite{Ekeland, ekeland_inverse} to derive a version of implicit function theorem for \eqref{eq:qcrit}.  While the use of the implicit function theorem to find critical points is quite well-known in the calculus of variations literature (see for instance \cite{LeDret} for a similar incompressible problem), the use of  EVP is much less common albeit its powerful nature. Our EVP-based approach is fully quantitative and does not require exactly inverting a linear operator.    Indeed, the allowance for error in our approach considerably simplifies the required calculations.  As long as $S$ is close enough to the identity in $W^{2,r}(\Omega)$, we are able to use EVP to produce a critical point $(Z^*, q^*)\in W^{2,r}_{\id}(\Omega)\times W^{1,r}(\Omega)$ with sufficient regularity to conclude that $Z^*$ is the unique minimizer.

  To illustrate the significance of our quantitative result beyond stationary variational problems, we apply our main result to develop a discrete-in-time minimizing movements scheme to generate mild solutions of the incompressible Navier Stokes equations.  Indeed, our particular interest in the $H^1$-projection problem is rooted in the connection to Navier-Stokes.  This connection can be traced back to Arnold's celebrated geometric interpretation of the incompressible Euler equations \cite{arnold}.   In \cite{brenier_combo}, Brenier gave a very concrete reinterpretation of Arnold's idea as a projection problem.  First one lets the fluid evolve for a short time taking into account inertia only (i.e. evolve the Lagrangian flow map $X$ by the equation $\partial_{tt} X=0$), then the resulting fluid configuration is then projected back onto the space $\mathcal D{\rm iff}(\Omega)$.  Given a sequence of fluid configurations $\{X_0, X_1, \ldots, X_n\}$ and a time step $\tau>0$, Brenier's scheme finds the next fluid configuration by solving the variational problem
\begin{equation}\label{eq:l2_proj}
X_{n+1}\in \argmin_{X\in D{\rm iff}(\Omega)} \frac{1}{2}\norm{\frac{X-X_n}{\tau}-\frac{X_{n}-X_{n-1}}{\tau}}_{L^2(\Omega)}^2. 
\end{equation}
The problem tries to find an incompressible map $X_{n+1}$, whose velocity $\frac{X_{n+1}-X_n}{\tau}$ best matches the velocity at the previous time step $\frac{X_{n}-X_{n-1}}{\tau}$, or in other words, the incompressible map with the least $L^2$ acceleration \cite{gangbo_westdickenberg}.  In fact, problem (\ref{eq:l2_proj}) can be viewed as the $L^2$ analogue of problem (\ref{eq:h1_proj}).

The $H^1$ projection problem appears when one wishes to extend the Brenier formalism to the Navier-Stokes equations \eqref{eq:NSE}.
 In Lagrangian coordinates, the no-slip Navier-Stokes equations take the form 
\begin{equation}\label{eq:lagrangian_ns}
\partial_{tt} X-\mu\Delta v(X)+\nabla p(X)=0, \quad \partial_t X=v(X),\quad \det(DX)=1, \quad X|_{\partial\Omega}=\id,
\end{equation}
where $\mu>0$ is a parameter that represents the viscosity of the fluid.  As one can see from the above equation, it is somewhat awkward to express viscous forces in Lagrangian coordinates.  For this reason, when given a sequence of fluid configurations $\{X_0,\ldots, X_n\}$, it is more natural to find the next fluid configuration $X_{n+1}$ by decomposing $X_{n+1}=Z_{n+1}\circ X_n$ and solving for $Z_{n+1}$.  The map $Z_{n+1}$ will be determined by solving the viscous analogue of Brenier's problem
\begin{equation}\label{eq:proj_1}
Z_{n+1}=\argmin_{Z\in  D{\rm iff}_{\id}(\Omega)}\frac{1}{2}\norm{\frac{Z\circ X_n-X_n}{\tau}-\frac{X_{n}-X_{n-1}}{\tau}}_{L^2(\Omega)}^2 +\frac{\mt}{2}\norm{\frac{DZ-I}{\tau}}_{L^2(\Omega)}^2,
\end{equation}
 where $I$ is the identity matrix.  In contrast to (\ref{eq:l2_proj}), Problem (\ref{eq:proj_1}) attempts to evolve the fluid by finding the incompressible map that simultaneously minimizes both the $L^2$ acceleration of the fluid and the viscous dissipation term $\frac{\mt}{2}\norm{\frac{DZ-I}{\tau}}_{L^2(\Omega)}^2$, which measures the instantaneous loss of kinetic energy to heat.   
 
 As one might expect, problem (\ref{eq:proj_1}) is equivalent to a special case of the $H^1$-projection problem. 
 Indeed, if one chooses $a=\mt$ and sets $S=\id+\tau v$, where $v$ solves
 \[
 (I-\mt\Delta) v=\big(\frac{X_n-X_{n-1}}{\tau}\big)\circ X_n^{-1}, \quad v|_{\partial\Omega}=0,
 \]
  then (\ref{eq:proj_1}) and (\ref{eq:h1_proj}) are identical up to an irrelevant constant term.  Note that the term $\big(\frac{X_n-X_{n-1}}{\tau}\big)\circ X_n^{-1}$  in the above equation roughly corresponds to the Eulerian vector field at time $n\tau$.  Hence, $S\in W^{2,r}(\Omega)$ roughly corresponds to the  vector field being an element of $L^r$.  As we shall show in the final section of the paper, our $W^{2,r}$ theory for the $H^1$ projection problem allows us to build short-time Eulerian and Lagrangian solutions to the Navier-Stokes equations starting from $L^r$ initial data.  Compared to the well-known work of Ebin and Marsden \cite{EM}, which also constructs Navier-Stokes solutions using the volume preserving diffeomorphism formalism,  we are able to solve the equations with much rougher initial data.

\section{Summary of main results}

In this subsection we give precise statements of the main results obtained in this paper.     We begin with our uniqueness result, which produces a sufficient condition for critical points $(Z^*, q^*)$ of (\ref{eq:qcrit}) to be the unique minimizer of the $H^1$ projection problem.  

\begin{thm}[Uniqueness]\label{thm:unique} Suppose that $d\in \{2,3\}$ and there exists a pair $(Z^*, q^*)\in W^{1,d}(\Omega)\times L^{\infty}(\Omega)$ satisfying (\ref{eq:qcrit}).    If there exists a constant $\sigma>0$ such that the singular values of $DZ^*$ are larger than $\sigma$ almost everywhere and 
\begin{equation}\label{eq:aug31.5}
\norm{q^*-\hat{q}}_{L^{\infty}(\Omega)} \leq a\sigma^2\big(2+3(1+\sqrt{3})\big)^{-1},  
\end{equation}
where $\hat{q}=\frac{1}{|\Omega|}\int_{\Omega} q(x)\, dx$,
then $Z^*$ is the unique global minimizer of (\ref{eq:h1_proj}).
\end{thm}

As we noted earlier, we shall prove Theorem \ref{thm:unique} by showing that the Lagrangian (\ref{eq:lagrangian}) has a novel convexity property.   In particular, we shall show that it is possible to control the concavity of the determinant term $\int_{\Omega} q(x)|\det(DZ(x))|\, dx$. If we let $\mathcal{B}(Z,Z_0, q)$ denote the {\it Bregman divergence} \cite{Bregman}
\begin{equation}\label{eq:bregman}
\mathcal{B}(Z,Z_0,q):=\int_{\Omega} q(x)\Big(|\det(DZ)|-|\det(DZ_0)|-\sgn(\det(DZ_0))\cof(DZ_0)):(DZ-DZ_0)\Big)(x) \, dx,
\end{equation}
then, when $d\leq 3$, we prove the inequality
\begin{equation}\label{eq:bregman_ineq}
-\mathcal{B}(Z,Z_0,q)\leq \frac{c}{2\sigma(DZ_0)^2}\norm{q\det(DZ_0)}_{L^{\infty}(\Omega)}\norm{DZ-DZ_0}_{L^2(\Omega)}^2\quad \hbox{ for } q\geq 0,
\end{equation}
where $c = 1+\frac{3}{2}(1+\sqrt{3})$ and $\sigma(DZ_0)$ is a lower bound on the smallest singular value of $DZ_0$.  Thus, we see that it is possible to control the concavity of $Z\mapsto \int_{\Omega} q(x)|\det(DZ(x))|\, dx$ with $H^1(\Omega)$, as long as we are at a base point $(Z_0,q)$ that is not too irregular.  Indeed, we shall obtain our uniqueness result by controlling the concavity of $Z\mapsto \int_{\Omega} q(x)|\det(DZ(x))|\, dx$ at the critical point $(Z^*, q^*)$ with the $H^1$ term $Z\mapsto \frac{a}{2}\norm{DZ-DS}_{L^2(\Omega)}^2$.  

 Beyond the application to the $H^1$ projection problem, the inequality (\ref{eq:bregman_ineq}) is useful for more general polyconvex variational problems in nonlinear elasticity.  This is due to the fact that terms of the form $\int_{\Omega} q(x)|\det(DZ(x))|\, dx$ can be made to appear in any polyconvex problem involving determinants.

Our uniqueness result is complemented by the following theorem, which guarantees the existence of critical points to (\ref{eq:qcrit}) that satisfy the conditions of Theorem \ref{thm:unique} when $S$ is close to the identity in $W^{2,r}(\Omega)$.  Taken together, Theorem \ref{thm:unique} and \ref{thm:main_2} guarantee  that the $H^1$ projection problem has a unique solution when $S$ is sufficiently close to the identity in $W^{2,r}(\Omega)$. 
\begin{thm}[Existence]\label{thm:main_2}
 For $d\in \{2,3\}$ and $r\in (d,\infty)$, let us define
 \begin{equation}\label{eq:intro_lin_proj}
 u^*:=\argmin_{u\in H^1_0(\Omega), \nabla \cdot u=0} \norm{S-\id-u}_{L^2(\Omega)}^2+a\norm{D(S-\id-u)}_{L^2(\Omega)}^2,
 \end{equation}
 and
  \begin{equation}\label{eq:intro_delta_defs}
\delta:=\norm{(I-a\Delta)(S-\id)}_{L^r(\Omega)}, \quad \delta':=\norm{(I-a\Delta)(S-\id-u^*)}_{L^r(\Omega)}.
 \end{equation}
 If $\delta$, $a^{-\frac{d+r}{2r}}\delta$, $a^{-1}\delta'$ and $a^{-\frac{d+3r}{2r}}\delta^2= a^{-\frac{d+r}{2r}}\delta (a^{-1} \delta)$ are sufficiently small,
 then there exists $Z^*\in W^{2,r}(\Omega)$ and $q^*\in W^{1, r}(\Omega)$ that satisfy the hypotheses of Theorem \ref{thm:unique}. Furthermore, 
 \[
\norm{(I-a\Delta)(Z^*-S)}_{L^r(\Omega)} \lesssim_{d,r} \delta'+ a^{-\frac{d+r}{2r}}\delta^2.
 \]
 \end{thm}
 \begin{remark}
The constants in above result can be derived explicitly from our arguments in Section \ref{sec:main_estimates}.
 \end{remark}
  \begin{remark}
One may wonder if a similar result holds when $W^{2,r}(\Omega)$ is replaced by a different Banach space $X$ of maps from $\Omega$ to $\Omega$ (if $S$ is close to the identity in $X$, can one find a critical point $Z^*\in X$?). In our argument we use three crucial properties of $W^{2,r}(\Omega)$, namely that $W^{2,r}(\Omega)$ embeds into $W^{1,\infty}(\Omega)$,   it is closed with respect to composition, and that $\Delta^{-1}\partial^2_{ij}$ is a bounded operator from $W^{2,r}(\Omega)$ to itself.  We suspect that this result would hold for any space $X$ with those three properties. However, we anticipate that the argument would need to overcome additional technical difficulties if $X$ is a space weaker than $W^{2,r}(\Omega)$. 
 \end{remark}

 In general, it is not so simple to estimate $\delta'$ from the data $S$ alone.  Nonetheless, $\delta'$ must always be bounded by a constant multiple of $\delta$ (c.f. Lemma \ref{stokes}), hence, one can also restate Theorem \ref{thm:main_2} in the following simpler but weaker form. 

 \begin{thm}\label{thm:main_1} 
 Suppose that $r>d\in \{2, 3\}$.
Define $\delta$ as in Theorem \ref{thm:main_2}.
 If $\delta$, $a^{-\frac{d+r}{2r}}\delta$ and $a^{-1}\delta$  are sufficiently small,
 then there exists $Z^*\in W^{2,r}(\Omega)$ and $q^*\in W^{1, r}(\Omega)$ that satisfy the hypotheses of Theorem \ref{thm:unique}.
 \end{thm}

Let us briefly discuss previously known results in the literature.   By appealing to abstract results in convex duality \cite{Preiss}, one can deduce that there exists a dense set $D\subset H^1(\Omega)$ such that the $H^1$ projection problem has at most one minimizer when $S\in D$.  Unfortunately, there is no known characterization of this set beyond its denseness, which limits its practical usefulness (for instance the interior of this set may be empty).  Furthermore, this result is silent on the question of existence.  In contrast, our result shows that a unique solution exists for maps $S$ in an entire ball around the identity in $W^{2,r}(\Omega)$.

Several authors have considered existence and uniqueness of minimizers to polyconvex problems of the form (\ref{eq:intro_functional}) in more concrete settings.  In three dimensions,  \cite{LeDret} studied the existence of regular critical points to (\ref{eq:intro_functional}) in the incompressible case $\mathcal{U}=\mathcal{D}\textrm{iff}_{\id}(\Omega)$ under the assumption that the applied force $F$ was small in an appropriate space.   Building on this, in \cite{Zhang}, Zhang showed that  when $j$ has the form  
\begin{equation}\label{eq:april5.2020.1}
j(x, M)= G(x, M)+ b\Big(|M|^r+ |{\rm cof}(M)|^s \Big), 
\end{equation}
for some polyconvex function $G:\Omega\times \RR^{d\times d}\to\RR$ and some parameters $b>0$, $r \geq 2$ and $s \geq r/(r-1)$, then the critical points from \cite{LeDret} are in fact the unique global minimizers of (\ref{eq:intro_functional}) provided that they satisfy certain norm bounds. The presence of the term $|{\rm cof}(M)|^s$ in \eqref{eq:april5.2020.1} was essential for \cite{Zhang}, where ${\rm cof}(M)$ denotes the cofactor matrix of $M$.   The resulting bound on the cofactor matrix allows a much better control over determinants, thanks to the fundamental identity $M^T\cof(M)=\det(M)I$.  Indeed, conditions in \cite{Zhang} imply that the determinant map $Z\mapsto \det(DZ)$ is weakly continuous along any minimizing seqnuence for (\ref{eq:intro_functional}).  Hence, the existence of minimizers for such functionals follows from the standard theory.   Clearly, these results do not apply to the $H^1$ projection problem:  the projection functional (\ref{eq:h1_proj}) does not afford any control on the cofactor matrix of $DZ$.

In \cite{GhoussoubKLP}, the notion of $\lambda$-convexity \cite{lambda} is used to provide a sufficient condition for a critical point of (\ref{eq:intro_functional}) to be the unique global minimizer of the problem.  While their result could be applied to the $H^1$-projection problem, it requires strong bounds on the optimal Lagrange multiplier $q^*$.  In particular,  it is necessary that the eigenvalues of $D^2 q^*$ are uniformly bounded from below.    Since there is no apparent mechanism in the $H^1$ projection problem that encourages $\lambda$-convexity of $q^*$, it seems unlikely that such a property can be obtained without showing that $q^*\in W^{2,\infty}(\Omega)$.  This is two full derivatives stronger than our condition, and hence, considerably more difficult to satisfy.  

\subsection{Applications to Navier-Stokes}

Finally, in the last section of this paper, we use the $H^1$ projection problem to construct solutions to no-slip Navier-Stokes equations:
\begin{equation}\label{eq:NSE}
\partial_t v - \mu\Delta v + v\cdot \nabla v + \nabla p = 0, \quad \nabla\cdot v = 0\,\, \hbox{ in } \Omega \times (0,T), \quad v=0\hbox{ on } \partial\Omega \times  (0,T),
\end{equation} 
with initial data $v_0 \in L^r(\Omega)$.

  For technical reasons, we shall use a slightly different scheme than the one given in (\ref{eq:proj_1}). 
Given an initial divergence free velocity $v_0$ and a time step $\tau$, we will construct discrete-in-time approximations to the Navier-Stokes equations using the $H^1$ projection problem by iterating  the following scheme:
\begin{equation}\label{eq:ns_scheme_step_1_intro}
(I-\mt\Delta)S_{n+1,\tau}=\id+\tau v_{n,\tau}, \quad S_{n+1,\tau}|_{\partial\Omega}=\id;
\end{equation}
  \begin{equation}\label{eq:ns_scheme_step_2_intro}
Z_{n+1,\tau}\in \argmin_{Z\in \mathcal{D}\textrm{iff}_{\id}(\Omega)} \frac{1}{2}\norm{Z-S_{n+1,\tau}}_{L^2(\Omega)}^2+\frac{\mt}{2}\norm{DZ-DS_{n+1}}_{L^2(\Omega)}^2;
\end{equation}
\begin{equation}\label{eq:ns_scheme_step_3_intro}
w_{n+1,\tau}:=  Z_{n+1\, \#}v_{n,\tau};
\end{equation}
\begin{equation}\label{eq:ns_scheme_step_4_intro}
v_{n+1,\tau}:= e^{-\mt \mathcal{A}}w_{n+1}.
\end{equation}
Here $\mathcal{A}$ is the so called \emph{Stokes operator}, and $v_{n+1,\tau}$ is the solution to the parabolic equation
\[
\partial_t v+\mathcal{A}v=0, \quad \nabla \cdot v=0, \quad v|_{\partial\Omega}=0,
\]
at time $\tau$ starting from the initial data $w_{n+1,\tau}$.  Using the scheme, we shall also define the Lagrangian flow maps
\[
X_{n+1,\tau}=Z_{n+1,\tau}\circ X_{n,\tau}.
\]
Our main result on Navier-Stokes can then be summarized as follows.
\begin{thm}\label{thm:NSE}
Let $v_0\in L^r(\Omega)$ with $r>d\in \{2,3\}$. Then there exists a time $T^*>0$ depending only on $\norm{v_0}_{L^r(\Omega)}, r, d$ and the viscosity $\mu$ in \eqref{eq:NSE} such that the following holds: 
 \begin{itemize}
 \item[(a)] The scheme (\ref{eq:ns_scheme_step_1_intro}-\ref{eq:ns_scheme_step_4_intro}) is well-defined and generates discrete velocities $v_{n,\tau}$ that are uniformly bounded in $L^r(\Omega)$ for $0\leq n\tau \leq T^*$.
 \item[(b)] The discrete velocities  converge in $L^2([0,T^*]\times \Omega))$ as $\tau\to 0$ to the unique mild solution $v\in L^{\infty}([0,T^*], L^r(\Omega))\cap L^1((0,T^*];W^{1,\infty}_0(\Omega))$ of the Eulerian Navier-Stokes equation \eqref{eq:NSE}.
 \item[(c)] The discrete Lagrangian flow maps converge  in $L^1([0,T^*]\times \Omega)$ as $\tau \to 0$ to the unique solution of the Lagrangian Navier-Stokes equation \eqref{eq:lagrangian_ns}. 
 \end{itemize}
 \end{thm}
The solution $v$ and $X$ in above theorem satisfies stronger regularity properties than those listed above: see Remark \ref{eq:last-remark}.

\section{A new sufficient conditions for being a minimizer: proof of Theorem \ref{thm:unique}}

To prove Theorem \ref{thm:unique}, we will show that our assumption on the pair $(Z^*, q^*)$ guarantees that $J_a$ lies above a convex parabola centered at $Z^*$ and touches the parabola at $Z^*$.  In other words, we will obtain the inequality $J_a(Z)\geq J_a(Z^*)+\frac{1}{2}\norm{Z-Z^*}_{L^2(\Omega)}$, see \eqref{eq:done}.    
The argument involves two key steps. First, we relax the constraint on $\det(DZ)$ by introducing $q^*$ as a Lagrange multiplier, see (\ref{eq:relaxation}).   We then show that the non-convexity of the Lagrange multiplier term $\int_{\Omega} q^*|\det(DZ)|$ is dominated by the quadratic term $\frac{a}{2}\norm{DZ-DS}_{L^2(\Omega)}^2$ by establishing the Bregman divergence bound in (\ref{eq:bregman_ineq}).   This bound is the consequence of an interesting matrix inequality that holds in dimensions 2 and 3 (Lemma \ref{lem:matrix_inequality}).  Though the matrix inequality is elementary, it plays an essential role in our argument that we believe is worth highlighting.

\begin{lemma}\label{lem:matrix_inequality}
Let $M, A\in \RR^{d\times d}$.
In two dimensions,
\[
\Big(\sgn(\det(A))\cof(A)-\sgn(\det(M))\cof(M)\Big):(A-M)\geq -\frac{|\det(M)|}{\sigma^2}|A-M|^2.
\]
In three dimensions,  
\[
\Big(\sgn(\det(A))\cof(A)-\sgn(\det(M))\cof(M)\Big):(A-M)\geq -\big(1+\frac{3}{2}(1+\sqrt{3})\big) \frac{|\det(M)|}{\sigma^2}|A-M|^2
\]
where $\sigma$ is the smallest singular value of $M$. 
\end{lemma}

\begin{remark}\label{rem:determinant1} 

 As a polynomial of degree $d$, the restriction of the determinant function to any convex subset of the set of $\RR^{d\times d}$, is $\lambda$-convex.  When $d=2$, we can choose $\lambda=-1$ independently of the convex set. This means, 
\[
\bigl( \cof(A)- \cof(M) \bigr): (A-M) \geq -|A-M|^2 \qquad A, M \in \RR^{2\times 2}.
\]
\end{remark}
\begin{remark}
 There is no analogous inequality when $d\geq 4$.  Indeed, if we choose $M=I$ and $A=\alpha I-(\alpha-\alpha^{1-d}) \bm{e}_d\otimes \bm{e}_d$, then as $\alpha\to \infty$, the left-hand-side of the inequality scales like $-\alpha^{d-1}$ while the right-hand-side scales like $-\alpha^2$. 
\end{remark}
\begin{proof}
Thanks to the density of diagonalizable nonsingular matrices, we can assume without loss of generality that $A$ and $M$ are nonsingular and diagonalizable.
Thus, we can factor $A=BM$ for some matrix $B$.  
We then have 
\[
\Big(\sgn(\det(A))\cof(A)-\sgn(\det(M))\cof(M)\Big):(A-M)=|\det(M)|\Big(\sgn(\det(B))\cof(B)-I\Big):(B-I).
\]
Expanding out the product and using the fact that $\cof(B)^T=\det(B)B^{-1}$, the right-hand-side is equal to
\[
|\det(M)|\big( d|\det(B)|+d-\tr(B)-|\det(B)|\tr(B^{-1})\big).
\]

Since $\sigma$ is the smallest singular value of $M$, we have
\[
|A-M|^2\geq \sigma^2|B-I|^2.
\]
Hence, given any constant $c>0$, we obtain 
\[
\Big(\sgn(\det(A))\cof(A)-\sgn(\det(M))\cof(M)\Big):(A-M) + c\frac{|\det(M)|}{\sigma^2}|A-M|^2\geq
\]
\[
|\det(M)|\big(c|B-I|^2+ d|\det(B)|+d-\tr(B)-|\det(B)|\tr(B^{-1})\big).
\]
It is now clear that the lower bound only depends on $\det(M)$ and the eigenvalues of $B$.  

Let us now show that 
\[
f(B):=c|B-I|^2+ d|\det(B)|+d-\tr(B)-|\det(B)|\tr(B^{-1})
\]
is nonnegative once $c$ is sufficiently large.  It is clear that $f$ is a function of the eigenvalues of $B$ and that it is minimized when the eigenvalues are nonnegative.  Therefore, we shall assume that $B$ is a diagonal matrix with nonnegative eigenvalues in the rest of the argument.

In two dimensions, $|\det(B)|\tr(B^{-1})=\tr(B)$.  It is then easy to check that $f$ has a single critical point at $B=I$.  When $c> 1$, $f$ is coercive, thus, $I$ must be the unique minimizer when restricted to diagonal matrices.  Since $f(I)=0$, we obtain the desired inequality in two dimensions by letting $c\to 1$.

In three dimensions, we have the inequality 
\[
|\det(B)|\tr(B^{-1})\leq |B-I|^2+2\tr(B)-3.
\]
Thus, 
\[
f(B)\geq g(B):=(c-1)|B-I|^2+3|\det(B)|-3\tr(B)+6.
\]
As long as $c\geq 1$, $g$ is coercive. 

Now consider $h(B):=(c-1)|B-I|^2-3\tr(B)+6$, which lies below $g$, and set  
\[
\bar h(t)=(c-1)|t-1|^2-3t+2, \qquad t \geq 0.
\]
When $c> 1$,  $h$ is  strictly convex and has a unique global minimum over diagonal matrices at its critical point $B_0:=aI$ where $a:=1+\frac{3}{2(c-1)}$. Similarly, $\bar h$ has a unique global minimizer at $a$. For 
$$B= {\rm diag}(x_1, x_2, x_3), \quad  \tilde B=:{\rm diag}\Bigl(\min\{a, x_1\}, \min\{a, x_2\}, \min\{a, x_3\}\Bigr)$$ 
with nonnegative entries, we have 
\[
h(\tilde B)= \sum_{i=1}^3\bar h\bigl(\min\{a, x_i\}\bigr) \leq \sum_{i=1}^3 \bar h(x_i)=h(B).
\]
The previous inequality is strict unless $\tilde{B}=B.$ Since $\det( \tilde B) \leq \det( B)$ we can conclude that $ g(\tilde B) \leq g(B)$ and again the latter inequality is strict unless $\tilde{B}=B.$ Therefore, any minimizer of $g$ must have eigenvalues bounded in $[0, 1+\frac{3}{2(c-1)}]$.

By direct calculation, the Hessian of $g$ is diagonally dominant when restricted to the set of diagonal matrices whose eigenvalues are bounded in $[0, \frac{c-1}{3}]$.  Thus, $g$ must be convex in this region.     This region is guaranteed to contain the minimizer of $g$ as soon as  $1+\frac{3}{2(c-1)}\leq \frac{c-1}{3}$ which is equivalent to $c\geq 1+\frac{3}{2}(1+\sqrt{3})$. In this case, the critical point of $g$ at $C=I$ must be the global minimum with value $g(I)=0$. Therefore, in three dimensions, it follows that
 \[
\Big(\sgn(\det(A))\cof(A)-\sgn(\det(M))\cof(M)\Big):(A-M) + \big(1+\frac{3}{2}(1+\sqrt{3})\big)\frac{|\det(M)|}{\sigma^2}|A-M|^2\geq 0,
\]
which is the desired result.
 \end{proof}

With the matrix inequality in hand, we can prove Proposition \ref{pro:sufficient-cond}, which establishes the uniqueness of minimizers for the Lagrangian relaxation (\ref{eq:lagrangian}).    This produces Theorem \ref{thm:unique} as an immediate consequence.

\begin{prop}\label{pro:sufficient-cond} Let $d\in \{2, 3\}$ and let
\begin{equation}\label{eq:relaxation}
\cL(Z, q):=\frac{1}{2}\norm{Z-S}_{L^2(\Omega)}^2+\frac{a}{2}\norm{DZ-DS}_{L^2(\Omega)}^2+\int_{\Omega} q(x)(|\det\big(DZ(x)\big)|-1)\, dx.
\end{equation}
Suppose that $Z^*\in \mathcal D{\rm iff}_{\id}(\Omega)$ and $q^* \in L^{\infty}(\Omega)$ solve (\ref{eq:qcrit}) and the singular values of $DZ^*$ are uniformly bounded from below by some $\sigma>0$.   Set $\bar{q}:=q^*-c$, where $c\in \RR$ is the largest constant so that $\bar{q}\geq 0$.
   If $\norm{\bar{q}}_{L^{\infty}}\leq a\sigma^2\big(1+\frac{3}{2}(1+\sqrt{3})\big)^{-1}$, then $Z^*$ is the unique global minimizer of $\cL(\cdot,\bar{q})$ among functions in $W^{1,d}_{\id}(\Omega)$. 
   \end{prop}
\begin{proof}  

Let $Z$ be some arbitrary element of $W^{1,d}_{\id}(\Omega)$.
  Calculating the $Z$ variation of $\cL$, we see that 
\[
\delta_Z \cL(Z, \bar q)(\phi) =\int_{\Omega} \bigg((Z-S)\cdot \phi +\Big( DZ-DS+ \bar q\sgn(\det(DZ))\cof(DZ)\Big): D\phi\bigg) dx,
\]
where $\phi\in W^{1,d}_0(\Omega)$ is an arbitrary perturbation.
Since $(Z^*,q^*)$ solves (\ref{eq:qcrit}), it follows that $\delta_Z \cL(Z^*, \bar q) \equiv 0$.   Hence,  
\[
\cL(Z, \bar q)-\cL(Z^*, \bar q)=\cL(Z, \bar q)-\cL(Z^*, \bar q)-\delta_Z \cL(Z^*, \bar q) (Z-Z^*)=\frac{1}{2}\norm{Z-Z^*}_{L^2(\Omega)}^2+\frac{a}{2}\norm{DZ-DZ^*}_{L^2(\Omega)}^2+\mathcal{B}(Z,Z^*,\bar{q}),
\]
where we recall the definition of $\mathcal{B}(Z,Z^*,\bar{q})$ from (\ref{eq:bregman}). Applying the Fundamental Theorem of Calculus, we have
\[
\mathcal{B}(Z,Z^*,\bar{q})=\int_{\Omega} \bar q(x)\int_0^1\frac{1}{t} \Big(\sgn\big(\det(DZ_t(x))\big)\cof\big(DZ_t(x)\big)- \cof\big(DZ^*(x)\big)\Big):\big(DZ_t(x)-DZ^*(x)\big) \, dx\, dt,
\]
where $Z_t=tZ+(t-1)Z^*$ and we have used the fact that $\frac{1}{t}(DZ_t-DZ^*)=(DZ-DZ^*)$.
Now Lemma \ref{lem:matrix_inequality} combined with the $L^{\infty}$ bound on $\bar{q}$ implies that
\begin{equation}\label{envelop}
 \cL(Z, \bar q) \geq \cL(Z^*, \bar q) + {1\over 2} \norm{Z-Z^*}_{L^2(\Omega)}^2.
\end{equation}

\end{proof}

\textbf{Proof of Theorem~\ref{thm:unique}}. 

Let $Z_0$ be some arbitrary element of $\mathcal{D}\textup{iff}_{\id}(\Omega)$.  Since $\det(DZ_0)=1$ everywhere, it follows that 
\[
\chi(Z_0)\geq \sup_{q\in L^1(\Omega)} \int_{\Omega} q(x)\big(\det(DZ_0(x))-1\big)\, dx=\sup_{q\in L^1(\Omega)} \int_{\Omega} q(x)\big(|\det(DZ_0(x))|-1\big)\, dx.
\]
Therefore, 
\[
J_a(Z_0)\geq \sup_{q\in L^1(\Omega)} \mathcal{L}(Z_0, q)\geq \mathcal{L}(Z_0, \bar{q}),
\]
where $\bar{q}$ is defined as in Proposition \ref{pro:sufficient-cond}.  Using Proposition \ref{pro:sufficient-cond} and the fact that $J_a(Z^*)=\mathcal{L}(Z^*, \bar{q})$, we can conclude that
\begin{equation}\label{eq:done}
J_a(Z_0)\geq J_a(Z^*)+\frac{1}{2}\norm{Z-Z_0}_{L^2(\Omega)}^2. 
\end{equation}

%
%
%
%
%
%
%

%
%
%
%
%
%
%
%
%
%
%
%
%
\section{Preliminaries for the proof of Theorem \ref{thm:main_2}}  

It remains to prove the existence of the solution pair $(Z^*, q^*)$ that satisfies the hypothesis of Theorem \ref{thm:unique}.  In this section, we will introduce a number of basic results that will play an important role in our subsequent analysis.  We begin by introducing a special operator that will allow us to simplify equation (\ref{eq:qcrit}).

{\bf Changing a base point for an operator.} Given any operator $L$ from a subset of functions on $\Omega$ to another subset of functions on $\Omega$, whenever $Z: \bar \Omega \rightarrow \bar \Omega$ is invertible, we define 
\[
L_Z(f):= L \Big(f \circ Z^{-1} \Big) \circ Z.
\]
The operator $L_Z$ can be expressed in terms of the pull--back operator.

{\bf Leray projection operator $\PP$.} We set 
\[
\mathcal V_\id:= \Big\{ w \in L^2(\Omega; \Rd)\; : (w, \nabla \phi)=0 \; \textup{for all} \; \phi\in C^{\infty}(\Omega; \RR) \Bigr\}
\]
The Leray projection $\mathbb P: L^2( \Omega, \mathbb R^d) \rightarrow  \mathcal V_\id$ is the orthogonal projection of $L^2(\Omega)$ onto $\mathcal V_\id$. 
When $\partial \Omega$ is of class $C^{1,1}$, we have from Theorem 1 of \cite{M18} (also \cite{Solo73}, \cite{Solo77}) that
\begin{equation}\label{eq:april26.2020.1}
\|\mathbb P(\phi)\|_{W^{l,r}(\Omega)} \lesssim  \|\phi\|_{W^{l,r}(\Omega)} \qquad \forall l \in \{0, \cdots, k\} \quad \forall \phi \in W^{k,r}(\Omega),\; \forall r \in (1, \infty).
\end{equation}


{\bf The Projection operator $\PP_Z$.} 
Given a map $Z\in \mathcal{D}\textup{iff}_{\id}(\Omega)$ we can introduce the operator $\PP_Z$ from the Leray projection, using the change of base point formula.  Note that  $\PP_Z$ can also be understood as an orthogonal projection.  
If we define the space
\[
\mathcal V_Z:= \Big\{ w \in L^2(\Omega; \Rd)\; : (w, \nabla \phi(Z))=0 \; \textup{for all} \; \phi\in C^{\infty}(\Omega; \RR) \Bigr\},
\]
one can readily check that $w  \in \mathcal V_Z$ if and only if $w \circ Z^{-1} \in \mathcal V_\id$.    It then follows that $\PP_Z$ is the orthogonal projection of $L^2(\Omega;\Rd)$ onto $\mathcal{V}_Z$.

Now that we have defined $\PP_Z$, we can use it to simplify equation (\ref{eq:qcrit}) by eliminating the pressure/Lagrange multiplier variable $q$. This is accomplished in the following lemma.
\begin{lemma}
If $Z\in\mathcal{D}\textup{iff}_{\id}(\Omega)$ solves the equation
\begin{equation}\label{eq:ppcrit}
\PP_Z(I-a\Delta)(Z-S)=0,
\end{equation}
then there exists $q:\Omega\to\RR$ such that $(Z,q)$ is a solution to equation (\ref{eq:qcrit}).  Furthermore, $q$ can be recovered explicitly from the formula
\[
\nabla q=-DZ^T(I-a\Delta)(Z-S).
\]
\end{lemma}
\begin{proof}

Since $\PP_Z$ is the orthogonal projection of $L^2(\Omega)$ onto the space $\mathcal{V}_Z$, the condition $\PP_Z (I-a\Delta)(Z-S)=0$ implies that
\[
\Big((I-a\Delta)(Z-S), v\circ Z\big)=0,
\]
for every divergence free vector field $v$ with zero normal component.
As we noted in the introduction, this is equivalent to the existence of a scalar function $q:\Omega\to\RR$ such that
\[
(I-a\Delta)(Z-S)+D^T(\cof(DZ)q)=0.
\]
To recover $q$, we use the fact that $D^T(\cof(DZ)q)=\cof(DZ)\nabla q$ and $DZ^T\cof(DZ)=\det(DZ)I=I$.
\end{proof}

In the process of finding solutions to (\ref{eq:ppcrit}), we shall need to be able to invert the equations
\[
\PP_Z (I-a\Delta) u=w, \quad \nabla \cdot u=0, \quad u|_{\partial\Omega}=0
\]
where $w\in \mathcal{V}_Z$ and $Z\in \mathcal{D}\textup{iff}_{\id}(\Omega)$ are given and $u$ is unknown. In the special case where $Z=\id$, this is known as the Stokes resolvent problem.   This problem plays an important role in the study of the Navier-Stokes equations and will reappear throughout the rest of our paper.

 {\bf The Stokes operator and the Stokes resolvent problem.}  For $1< p<\infty$, let
$$
K_p:= \{u \in W^{2, p}(\Omega, \mathbb R^d) \cap W^{1, p}_0(\Omega, \mathbb R^d): \nabla \cdot u=0 \hbox{ in } \Omega \}.
$$ 
The {\it Stokes operator} 
\begin{equation}\label{stokes_op}
\mathcal{A}:=-\mathbb P \Delta: K_p\to L^p(\Omega)
\end{equation}
is defined to be the negative of the composition of the Leray projection and the Laplace operator. For well-posedness and regularity properties of this operator, see for instance \cite{RTemam},\cite{Giga}.  Using the Stokes operator, we can rewrite the {\it Stokes resolvent problem} for  a given $w$ as follows:
\begin{equation} \label{eq:stokes_r}
(I+a\mathcal{A})u+w=0, \quad \nabla \cdot u=0, \quad u|_{\partial\Omega}=0.
\end{equation}
 The following Lemma on the solvability and regularity of the Stokes resolvent problem will be essential to our critical point analysis, and will reappear again when we consider the Navier-Stokes equations. 
 \begin{lemma} [Theorem 1.2, \cite{FarwigSohr}] \label{stokes}
If $w\in L^p(\Omega)$, then there exists a solution $u\in K_p$ to equation (\ref{eq:stokes_r}), a scalar function $f:\Omega\to\RR$ and a constant $\bar{C}_p$ such that
\[
(I-a\Delta)u+\nabla f+w=0
\]
and
\begin{equation}\label{eqn:sobolev3}
\|u \|_{L^p(\Omega)} + a  \| D^2 u \|_{L^{ p}(\Omega)} +\| \nabla f\|_{L^{ p}(\Omega)} \leq  \bar{C}_p \|w \|_{L^p(\Omega)}.
\end{equation}
\end{lemma}

\section{EVP and the main ideas for the proof of Theorem \ref{thm:main_2}}

 Now we have converted the critical point equation (\ref{eq:qcrit}) into the simplified form (\ref{eq:ppcrit}). As mentioned in the introduction, we will now  use a version of the implicit function theorem based on Ekeland's variational principle (EVP).  These ideas are first introduced in an abstract setting.
 
 \subsection{EVP and the implicit function theorem}

\begin{prop}[Ekeland variational principle]\label{lem:Ekeland} Let $(\mathcal T, {\rm dist})$ be a complete metric space and let $F:\mathcal T \to \mathbb {R} \cup \{+\infty \}$ be a lower semicontinuous function that is bounded below and is not identically $\infty.$ If  $x_{0}\in \mathcal T$ such that $F(x_{0})\leq\epsilon +\inf_{\mathcal T} F$ for some $\epsilon>0$, then for all $\lambda>0$ there exists $ x_{\lambda}\in \mathcal T$ such that 
\[
F(x_{\lambda})\leq F(x_{0}), \quad {\rm dist}(x_0, x_{\lambda}) \leq \lambda, \qquad \text{and} \qquad  F(x_{\lambda})<F(x)+{\epsilon \over \lambda} {\rm dist}(x, x_{\lambda}) \qquad \forall x \in \mathcal T\setminus\{x_{\lambda}\}.
\]
\end{prop}    In order to use EVP to solve (\ref{eq:ppcrit}), we need to convert the question of finding zeros into a variational problem.  This is accomplished in the following abstract  lemma, which is an adaptation of Ekeland's argument from \cite{ekeland_inverse} that is well-suited to our setting.  
\begin{lemma}\label{lem:abstract_slope_bound}
Suppose that $\mathcal{X}, \mathcal{Y}$ are Banach spaces and $\Phi:\mathcal{X}\to\mathcal{Y}$ is a continuous and Frechet differentiable map.  Given a closed proper subset $\mathcal{M}\subset \mathcal{X}$,  we define a function $F:\mathcal{X}\to\RR\cup\{+\infty\}$ such that
\[
F(x):=\begin{cases}
\norm{\Phi(x)}_{\mathcal{Y}} &\textup{if}\; x\in \mathcal{M},\\
+\infty &\textup{otherwise}
\end{cases}
\]
Given a point $x_0\in M$ and $\lambda>0$, let $x_{\lambda}\in\mathcal{X}$ be the point provided by Ekeland's variational principle such that 
\[
F(x_{\lambda})\leq F(x_{0}), \quad \norm{x_0-x_{\lambda}}_{\mathcal{X}} \leq \lambda, \qquad \text{and} \qquad  F(x_{\lambda})<F(x)+{F(x_0) \over \lambda} \norm{x-x_{\lambda}}_{\mathcal{X}} \qquad \forall x \in  \mathcal{X}\setminus\{x_{\lambda}\}.
\]
If $F(x_{\lambda})\neq 0$ and $\gamma:[0,1]\to \mathcal{M}$ is a $C^1$ path such that $\gamma(0)=x_{\lambda}$, $\gamma'(0)=F(x_{\lambda})v$ for some vector $v\in \mathcal{X}$, then
\begin{equation}\label{eq:abstract_slope}
-1+\norm{\frac{\Phi(x_{\lambda})}{F(x_{\lambda})}+\textup{d}\Phi(x_{\lambda};v)}_{\mathcal{Y}} \geq -\frac{F(x_0)}{\lambda}\norm{v}_{\mathcal{X}},
\end{equation}
where $\textup{d}\Phi(x_{\lambda};v)$ is the Frechet derivative of $\Phi$ at $x_{\lambda}$ in the direction of $v$. 
\end{lemma}

\begin{proof}
Let us first note that it is valid to apply EVP to $F$, since $F$ is nonnegative, lower semicontinuous, and not identically infinity. 
EVP implies that
\[
F(\gamma(t))-F(x_{\lambda})> -\frac{F(x_0)}{\lambda}\norm{\gamma(t)-x_{\lambda}}_{\mathcal{X}}
\]
for all $t>0$.  By the triangle inequality
\[
F(\gamma(t))\leq (1-t)\norm{\Phi(x_{\lambda})}_{\mathcal{Y}}+t\norm{\Phi(x_{\lambda})+F(x_{\lambda})\textup{d}\Phi(x_{\lambda},v)}_{\mathcal{Y}}+\norm{\Phi(\gamma(t))-\Phi(x_{\lambda})-tF(x_{\lambda})\textup{d}\Phi(x_{\lambda}; v)}_{\mathcal{Y}}.
\]
Since $\Phi$ is Frechet differentiable, we have 
\[
\lim_{t\to 0} t^{-1}\norm{\Phi(\gamma(t))-\Phi(x_{\lambda})-tF(x_{\lambda})\textup{d}\Phi(x_{\lambda}; v)}_{\mathcal{Y}}=0.
\]
Therefore, 
\[
\lim_{t\to 0^+} \frac{F(\gamma(t))-F(x_{\lambda})}{t}\leq -F(x_{\lambda})+\norm{\Phi(x_{\lambda})+F(x_{\lambda})\textup{d}\Phi(x_{\lambda};v)}_{\mathcal{Y}}.
\]
Hence, it follows that 
\[
-F(x_{\lambda})+\norm{\Phi(x_{\lambda})+F(x_{\lambda})\textup{d}\Phi(x_{\lambda};v)}_{\mathcal{Y}}\geq -F(x_{\lambda})\frac{F(x_0)}{\lambda}\norm{v}_{\mathcal{X}}.
\]
Dividing both sides by $F(x_{\lambda})$ gives the result. 
\end{proof}
It is not immediately obvious how one can use Lemma \ref{lem:abstract_slope_bound} to find zeros of a map $\Phi$.   However, note that because $\norm{\frac{\Phi(x_{\lambda})}{F(x_{\lambda})}}_{\mathcal{Y}}=1$, the Lemma essentially gives a bound on the steepest descent rate of $F$ at $x_{\lambda}$ when $F(x_{\lambda})\neq 0$.  If we can show that this bound is impossible for some $\lambda>0$, then it follows that $F(x_{\lambda})=0$ and hence $\Phi(x_{\lambda})=0$.   For example, under the usual assumptions for the implicit function theorem (i.e. $\mathcal{M}=\varnothing$ and $v\mapsto \textup{d}\Phi(x, v)$ is a linear bijection with a uniformly continuous inverse for all $x$ in a neighborhood of $x_0$), we can choose $v=-\textup{d}\Phi(x_{\lambda}, \frac{\Phi(x_{\lambda})}{F(x_{\lambda})})^{-1}$, which is the steepest descent direction for $F$ at $x_{\lambda}$.   With this choice, the slope inequality will fail as long as $F(x_0)$ is sufficiently small and $\lambda$ is chosen appropriately.   

On the other hand, there is no reason that one needs to invert $v\mapsto \textup{d}\Phi(x, v)$ exactly.  As long as we can find a (valid) direction $v$ where the inequality (\ref{eq:abstract_slope}) fails, we will have found a zero of $\Phi$.  Indeed, this is the advantage of the EVP based approach --- we are allowed to make some error when we attempt to invert  $v\mapsto \textup{d}\Phi(x, v)$.  Furthermore, when we make a choice for $\lambda$ we will have the guarantee that the solution $x_{\lambda}$ is at most distance $\lambda$ away from the starting point $x_0$ in the $\mathcal{X}$ norm.  This gives us complete quantitative control on the solution.  Finally, this approach makes it very convenient to enforce a nonlinear constraint on the solution set.  If $\mathcal{M}\neq \varnothing$, then one just needs to ensure that the descent direction $v$ is chosen to be in the ``tangent space'' of $\mathcal{M}$ at $x_{\lambda}$. 

\subsection{Adapting the arguments to our setting}

To apply Lemma \ref{lem:abstract_slope_bound} to find zeros of (\ref{eq:ppcrit}), we need to give appropriate choices for the spaces $\mathcal{X}, \mathcal{Y}, \mathcal{M}$ and the map $\Phi$.  Once these have been chosen, we shall define $F$ as in Lemma \ref{lem:abstract_slope_bound}.  

We shall take $\mathcal{X}=W^{2,r}_{\id}(\Omega)$ with a modified norm that depends on the parameter $a>0$.  More precisely we take 
\begin{equation}\label{spaces}
\mathcal{X}=\mathcal{X}_a;\,\,\mathcal{Y}=L^r(\Omega);\,\,\hbox{ and }\mathcal{M}=W^{2,r}_{\id}(\Omega)\cap \mathcal{D}\textup{iff}_{\id}(\Omega),
\end{equation}
where $\mathcal{X}_a$ has its elements the same as $W^{2,r}_{\id}(\Omega)$ with the norm
\begin{equation}
\norm{Z}_{\mathcal{X}_a}:=\norm{Z}_{L^r(\Omega)}+a\norm{D^2 Z}_{L^r(\Omega)}. 
\end{equation}
  Note that by setting $\mathcal{M}=W^{2,r}_{\id}(\Omega)\cap \mathcal{D}\textup{iff}_{\id}(\Omega)$ we will ensure that any points produced by EVP will satisfy the determinant constraint $\det(DZ)=1$.   Finally, since we wish to solve (\ref{eq:ppcrit}), we shall define
\begin{equation}\label{eq:phi_def}
\Phi(Z):=\PP_Z(I-a\Delta)(Z-S),
\end{equation}
which is clearly a map from $W^{2,r}(\Omega)$ into $L^r(\Omega)$.  More precisely, $\PP_Z$ is a map from $W^{2,r}(\Omega)$ into $\mathcal{V}_Z\cap L^r(\Omega)$.   

To contradict the inequality (\ref{eq:abstract_slope}), we shall need to minimize $\norm{\frac{\Phi(Z_{\lambda})}{F(Z_{\lambda})}+\textup{d}\Phi(Z_{\lambda};v)}_{L^r(\Omega)}$.  Note that $\frac{\Phi(Z_{\lambda})}{F(Z_{\lambda})}$ must take the form
\begin{equation}\label{eq:w_def}
\frac{\Phi(Z_{\lambda})}{F(Z_{\lambda})}=w(Z_{\lambda})
\end{equation}
where $w\in L^r(\Omega)\cap \mathcal{V}$ is a divergence free vector field such that $\norm{w}_{L^r(\Omega)}=1$.  Hence, we must be able to find a solution $v$ that approximately solves the equation
\begin{equation}\label{eq:ift_target}
\textup{d}\Phi(Z_{\lambda};v)=-w(Z_{\lambda})
\end{equation}
for a given divergence free vector field $w$ with unit $L^r$ norm. 

Luckily, $\Phi(Z)$ is very nearly a linear map, the only nonlinear behavior comes from the operator $\PP_Z$.  Hence, apart from the contribution coming from $\PP_Z$, the Frechet derivative of $\Phi$ is trivial.   Given a point $Z\in\mathcal{M}$, let 
\[
d\PP(Z;v)=\lim_{t\to 0^+} \frac{\PP_{Z+tv}-\PP_{Z}}{t}
\]
 denote the Frechet derivative of $\PP_Z$ in the direction of a vector $v\in W^{2,r}(\Omega)$. We can then write equation (\ref{eq:ift_target}) as
 \begin{equation}\label{eq:gat_phi}
 \textup{d}\Phi(Z;v)=\PP_Z (I-a\Delta)v+d\PP(Z;v)(I-a\Delta)(Z-S).
 \end{equation}
The second term in the Frechet derivative of $\Phi$ is rather annoying to work with. Thus, rather than try to invert the full expression (\ref{eq:gat_phi}), we will just treat the second term as an error term and try to approximately solve
\begin{equation}\label{eq:simplified_target}
\PP_Z (I-a\Delta)v=-w(Z_{\lambda}).
\end{equation}
However, even this simplified expression is tricky to solve explicitly due to the combination of the operators $\PP_Z$ and $(I-a\Delta)$.  Indeed, $\PP_Z$ is a linear operator with base point $Z$, while $(I-a\Delta)$ is a linear operator with base point at the identity, thus their composition is rather complicated.   To simplify matters, we shall let $u$ be a solution to the Stokes resolvent problem
\[
(I+\mathcal{A}) u=-w, \quad \nabla \cdot u=0, \quad u|_{\partial\Omega}=\id
\]
and choose $v=u\circ Z_{\lambda}$.  This choice of $v$ will not exactly solve (\ref{eq:simplified_target}), hence, this leads to a second source of error that we shall also need to control. 

The above considerations are now summarized in the following Proposition, which simplifies Lemma \ref{lem:abstract_slope_bound} and converts it into our specific setting.
\begin{prop}\label{prop:specific_slope_bound}
Given a point $\tilde{Z}\in\mathcal{M}$ and some $\lambda>0$, let $Z_{\lambda}$ be the point chosen by Ekeland's variational principle starting from $\tilde{Z}$. If $F(Z_{\lambda})\neq 0$, then for $w$ given in (\ref{eq:w_def}) we have
\begin{equation}\label{eq:specific_slope_ineq0}
-1+a\norm{\PP_{Z_{\lambda}}\big(\Delta(u\circ Z_{\lambda})-(\Delta u)\circ Z_{\lambda}\big)}_{L^r(\Omega)}+\norm{d\PP(Z_{\lambda};u\circ Z_{\lambda})}_{r,r}\norm{(I-a\Delta)(Z_{\lambda}-S)}_{L^r(\Omega)}\geq -\frac{F(\tilde{Z})}{\lambda}\norm{u\circ Z_{\lambda}}_{\mathcal{X}_a},
\end{equation}
where $u$ solves the Stokes resolvent problem \eqref{eq:stokes_r} and
\begin{equation}
\norm{d\PP(Z_{\lambda};u\circ Z_{\lambda})}_{r,r} :=\sup_{\norm{f}_{L^r(\Omega)}\leq 1}\norm{d\PP(Z_{\lambda};u\circ Z_{\lambda})f}_{L^r(\Omega)}.
\end{equation}
\end{prop}
\begin{proof}
$u$ is a divergence free vector field vanishing on $\partial\Omega$.
Therefore, thanks to the construction in Appendix \ref{sec:standard}, there exists a $C^1$ curve $Z(t):[0,1]\to\mathcal{M}$ such that $Z(0)=Z_{\lambda}$ and $Z'(0)=u(Z_{\lambda})$.   Now we can apply Lemma \ref{lem:abstract_slope_bound} to obtain the inequality
\[
-1+\norm{w\circ Z_{\lambda}+\textup{d}\Phi(Z_{\lambda};u\circ Z_{\lambda})}_{L^r(\Omega)} \geq -\frac{F(\tilde{Z})}{\lambda}\norm{u\circ Z_{\lambda}}_{\mathcal{X}_a}.
\]
Using equation (\ref{eq:gat_phi}), the triangle inequality, and the definition of the operator norm $\norm{d\PP(Z_{\lambda};u\circ Z_{\lambda})}_{r,r}$, it follows that 
\[
-1+\norm{w\circ Z_{\lambda}+\PP_{Z_{\lambda}}(I-a\Delta)(u\circ Z_{\lambda})}_{L^r(\Omega)} +\norm{d\PP(Z_{\lambda};u\circ Z_{\lambda})}_{r,r}\norm{(I-a\Delta)(Z_{\lambda}-S)}_{L^r(\Omega)}\geq -\frac{F(\tilde{Z})}{\lambda}\norm{u\circ Z_{\lambda}}_{\mathcal{X}_a}.
\]
Finally, we note that
\[
-w\circ Z_{\lambda}=\big((I+a\mathcal{A})u\big)\circ Z_{\lambda}=\PP_{Z_{\lambda}}\big[(I-a\Delta) u\circ Z_{\lambda}\big].
\]
Thus, 
\[
w\circ Z_{\lambda}+\PP_{Z_{\lambda}}(I-a\Delta)(u\circ Z_{\lambda})=a\PP_{Z_{\lambda}}\big(\Delta(u\circ Z_{\lambda})-(\Delta u)\circ Z_{\lambda}\big).
\]

\end{proof}

\section{Estimates and the proof of Theorem~\ref{thm:main_2}}\label{sec:main_estimates}

In this section, we will complete the proof of Theorem \ref{thm:main_2} by estimating the various quantities in (\ref{eq:specific_slope_ineq0}) and choosing an appropriate starting point $\tilde{Z}$.  

\subsection{Estimates}
We begin by estimating the operator norm $\norm{d\PP(Z_{\lambda};u\circ Z_{\lambda})}_{r,r}$.    We will do this by estimating the difference
\begin{equation*}
\norm{\PP_{Z_1}-\PP_{Z_2}}_{r,r}
\end{equation*} 
for arbitrary maps $Z_1, Z_2\in\mathcal{D}\textup{iff}_{\id}(\Omega)$.  To start, we will consider the case where one of the maps is the identity.

\begin{lemma}\label{lem:april12.2020}
If $Z \in \mathcal D{\rm iff}_{\id}(\Omega)$,  then for $r\in (1,\infty)$
\[
\norm{\PP_Z-\PP}_{r,r}\leq \big(  \norm{I-DZ}_{L^{\infty}(\Omega)}+\norm{\cof(DZ)-I}_{L^{\infty}(\Omega)}\big)\norm{\PP}_{r,r}^2.
\]

\end{lemma}

\begin{proof}  Fix some function $f\in L^r(\Omega)$ and let $\xi$ be a smooth test function. We use the Hodge decomposition to write 
$$f=w+\nabla \vp, \quad w:=\PP f, \qquad \xi=\zeta+\nabla \psi, \quad \zeta:=\PP \xi.$$ 
Now if we test $(\PP_Z-\PP)f$ against $ \xi\circ Z$ we have 
\[
(\PP_Z f-\PP f,  \xi\circ Z)=(f,  \zeta(Z))-(w, \xi(Z))
\]
If we expand $f$ and $\xi$ in terms of their decompositions, the term  $ (w, \zeta(Z))$ appears in both expressions, so we arrive at
\[
 (\nabla \vp, \zeta(Z))-(w,  \nabla\psi(Z))
\]

Now we estimate each term separately.  Pushing forward by $Z$ we see that
\[
(\nabla \vp, \zeta(Z))=(\nabla \vp(Z^{-1}) ,  \zeta).
\]
The closely related quantity, $(\nabla (\vp\circ Z^{-1}), \zeta)$, vanishes. From the fact $\nabla (\vp\circ Z^{-1})=\cof(DZ)\nabla \vp(Z^{-1})$ we see that
\[
(\nabla \vp(Z^{-1}) ,  \zeta)=((I-\cof(DZ))\nabla \vp(Z^{-1}) ,  \zeta)
\]
Similar arguments reveal that
\[
-(w,  \nabla\psi(Z))=-\big(w, ( I-DZ^T) \nabla \psi(Z) \big)
\]
Therefore
\[
(\PP_Z f-\PP f,  \xi\circ Z)\leq \Big(\norm{I-DZ}_{L^{\infty}(\Omega)}+\norm{\cof(DZ)-I}_{L^{\infty}(\Omega)}\Big)\norm{f}_{L^r(\Omega)}\norm{\xi}_{L^{r'}(\Omega)}\norm{\PP}_{r,r}\norm{\PP}_{r',r'}\]
Since $\PP$ is self adjoint, by duality, $\norm{\PP}_{r,r}=\norm{\PP}_{r',r'}$.  $f$ and $\xi$ were arbitrary, so we can conclude the result.\end{proof}

\begin{corollary}\label{cor:lipschitz_estimate}
Suppose $Z_i \in \mathcal D{\rm iff}_{\id}(\Omega)$ for $i=1,2$. Then for $r\in (1,\infty)$ we have 
\[
\norm{\PP_{Z_1}-\PP_{Z_2}}_r\leq \Big( \norm{DZ_1\cof(DZ_2)^T-I}_{L^{\infty}(\Omega)}+\norm{\cof(DZ_1)DZ_2^T-I}_{L^{\infty}(\Omega)}\Big)\norm{\PP}_{r,r}^2.
\]
Furthermore, if $u\in C^1_0(\Omega)$ is divergence free, then
\[
\norm{d\PP(Z;u\circ Z)}_{r,r}\leq 2\norm{\PP}_{r,r}^2\norm{Du}_{L^{\infty}(\Omega)}
\]
for any $Z\in \mathcal D{\rm iff}_{\id}(\Omega)$.
\end{corollary}
\begin{proof}  

 Fix some function $f\in L^r(\Omega)$ and let $g=f\circ Z_2$.   Writing things in terms of $g$ we have $\PP_{Z_1}f=(\PP (g\circ Z_2\circ (Z_1)^{-1}))\circ Z_1$ and $\PP_{Z_2}f=(\PP g)\circ Z_2$.  Therefore,
 \[
 \norm{\PP_{Z_2}f-\PP_{Z_1}f}_{L^r(\Omega)}=\norm{(\PP g-\PP_{Y}g)\circ Z_2}_{L^r(\Omega)}
 \]
 where $Y=Z_1\circ Z_2^{-1}$.  Since $Z_2$ is measure preserving, we have 
 \[
 \norm{(\PP g-\PP_{Y}g)\circ Z_2}_{L^r(\Omega)}=\norm{\PP g-\PP_{Y}g}_{L^r(\Omega)}. 
 \]
 From the previous Lemma we get the bound 
 \[
 \norm{\PP g-\PP_{Y}g}_{L^r(\Omega)}\leq \big(\norm{DY-I}_{L^{\infty}(\Omega)}+\norm{\cof(DY)-I}_{L^{\infty}(\Omega)}\big)\norm{g}_{L^r(\Omega)}\norm{\PP}_{r,r}^2.
 \]
 We can then compute $DY(Z_2)=DZ_1\cof(DZ_2)^T$ and $\cof(DY(Z_2))=\cof(DZ_1)DZ_2^T$.  Recalling that $g=f\circ Z_2^{-1}$, we can conclude that 
 \[
 \begin{split}
& \big(\norm{DY-I}_{L^{\infty}(\Omega)}+\norm{\cof(DY)-I}_{L^{\infty}(\Omega)}\big)\norm{g}_{L^r(\Omega)}\\
\leq &
  \big(\norm{DZ_1\cof(DZ_2)^T-I}_{L^{\infty}(\Omega)}+\norm{\cof(DZ_1)DZ_2^T-I}_{L^{\infty}(\Omega)}\big)\norm{f}_{L^r(\Omega)}.
  \end{split}
 \]
Since $f$ is arbitrary, we can conclude the first result.

For the second result, using Appendix \ref{sec:standard}, we can construct a $C^1$ curve  $Z(t):[0,1]\to\mathcal{M}$ such that $Z(0)=Z$ and $Z'(0)=u\circ Z$.  We then have 
\[
\norm{\PP_{Z(t)}-\PP_{Z}}_r\leq \Big( \norm{DZ(t)\cof(DZ)^T-I}_{L^{\infty}(\Omega)}+\norm{\cof(DZ(t))DZ^T-I}_{L^{\infty}(\Omega)}\Big)\norm{\PP}_{r,r}^2.
\]
Therefore, 
\[
\norm{d\PP(Z;u\circ Z)}_{r,r}\leq\lim_{t\to 0^+} t^{-1}\Big( \norm{DZ(t)\cof(DZ)^T-I}_{L^{\infty}(\Omega)}+\norm{\cof(DZ(t))DZ^T-I}_{L^{\infty}(\Omega)}\Big)\norm{\PP}_{r,r}^2.
\]
We can then write
\[
DZ(t)=DZ+tDu(Z)DZ+o(t),
\]
and
\[
\cof(DZ(t))DZ^T-I=\cof(DZ(t))\big(DZ^T-DZ(t)^T\big).
\]
Thus, 
\[
\lim_{t\to 0^+} t^{-1}\Big( \norm{DZ(t)\cof(DZ)^T-I}_{L^{\infty}(\Omega)}+\norm{\cof(DZ(t))DZ^T-I}_{L^{\infty}(\Omega)}\Big)=2\norm{Du}_{L^{\infty}(\Omega)},
\]
and the second result now follows.

\end{proof}

We will use the following lemma to estimate the remaining terms in (\ref{eq:specific_slope_ineq}) involving $u$.

\begin{lemma}\label{lem:map_d_compose}
If $f\in W^{2,r}(\Omega)$ is a scalar function and $Z\in\mathcal{M}$ then for any indices $1\leq i,j\leq d$, we have
\[
\norm{\partial_{i,j}^2 (f\circ Z_{\lambda})-(\partial^2_{i,j} f)\circ Z_{\lambda}}_{L^r(\Omega)}\leq \norm{\nabla f}_{L^{\infty}(\Omega)}\norm{\partial^2_{i,j} Z}_{L^r(\Omega)}+\norm{D^2 f}_{L^r(\Omega)}\norm{\partial_i Z\otimes \partial_j Z-e_i\otimes e_j}_{L^{\infty}(\Omega)}
\]
and we have 
\[
\norm{\Delta (f\circ Z_{\lambda})-(\Delta f)\circ Z_{\lambda}}_{L^r(\Omega)}\leq \norm{\nabla f}_{L^{\infty}(\Omega)}\norm{\Delta Z}_{L^r(\Omega)}+\norm{D^2 f}_{L^r(\Omega)}\norm{DZDZ^T-I}_{L^{\infty}(\Omega)}
\]

where $e_i$ is the $i^{th}$ standard basis vector.
\end{lemma}
\begin{proof}
Computing directly, we have 
\[
\partial_{i,j}^2 (f\circ Z)=D^2f(Z):\partial_i Z\otimes \partial_j Z+\nabla f(Z)\cdot \partial^2_{i,j} Z.
\]
Writing $\partial_{i,j}^2 f=D^2 f:e_i\otimes e_j$, we see that 
\[
\partial_{i,j}^2 (f\circ Z)-(\partial^2_{i,j} f)\circ Z=D^2f(Z):(\partial_i Z\otimes \partial_j Z-e_i\otimes e_j)+\nabla f(Z)\cdot \partial^2_{i,j} Z.
\]
Hence, 
\[
\Delta (f\circ Z)-(\Delta f)\circ Z=D^2f(Z):(DZDZ^T-I)+\nabla f(Z)\cdot \Delta Z
\]
The result now follows from Holder's inequality and the fact that $Z$ is a measure preserving map.
\end{proof}

We can now state a version of Proposition \ref{prop:specific_slope_bound} that eliminates the dependence on the Stokes resolvent solution $u$.  Recall  that $\mathcal{M}$ is given in \eqref{spaces}.
\begin{prop}\label{prop:simplified_slope_bound}
Given a point $\tilde{Z}\in\mathcal{M}$ and some $\lambda>0$, let $Z_{\lambda}$ be the point chosen by Ekeland's variational principle starting from $\tilde{Z}$. Define 
\[
K_a:=\sup_{f\in W^{2,r}(\Omega)}\frac{\norm{Df}_{L^{\infty}(\Omega)}}{\norm{f}_{\mathcal{X}_a}}.
\]
and
\[
C_r:=\max(\norm{\PP}_{r,r}, \bar{C}_r),
\]
where $\bar{C}_r$ is the constant in (\ref{stokes}).
If $F(Z_{\lambda})\neq 0$, then
\begin{multline}\label{eq:specific_slope_ineq}
-1+C_r^2\Big(\norm{DZ_{\lambda}DZ_{\lambda}^T-I}_{L^{\infty}(\Omega)}+aK_a\norm{\Delta Z_{\lambda}}_{L^r(\Omega)}\Big)+2K_aC_r^3\norm{(I-a\Delta)(Z_{\lambda}-S)}_{L^r(\Omega)}\geq \\
-C_r\frac{F(\tilde{Z})}{\lambda}\Big(1+\sum_{i,j=1}^d\norm{\partial_i Z_{\lambda}\otimes \partial_j Z_{\lambda}-e_i\otimes e_j}_{L^{\infty}(\Omega)}+aK_a\norm{D^2 Z_{\lambda}}_{L^r(\Omega)}\Big)
\end{multline}
\end{prop}
\begin{proof}
Let $u$ and $w$ be defined as in  Proposition \ref{prop:specific_slope_bound}, and recall that we have the inequality
\[
-1+a\norm{\PP}_{r,r}\norm{\Delta(u\circ Z_{\lambda})-(\Delta u)\circ Z_{\lambda}}_{L^r(\Omega)}+\norm{d\PP(Z_{\lambda};u\circ Z_{\lambda})}_{r,r}\norm{(I-a\Delta)(Z_{\lambda}-S)}_{L^r(\Omega)}\geq -\frac{F(\tilde{Z})}{\lambda}\norm{u\circ Z_{\lambda}}_{\mathcal{X}_a}.
\]
Since
\[
\norm{u\circ Z_{\lambda}}_{\mathcal{X}_a}=\norm{u\circ Z_{\lambda}}_{L^r(\Omega)}+a\norm{D^2(u\circ Z_{\lambda})}_{L^r(\Omega)},
\]
we can use the measure preserving property of $Z_{\lambda}$ and the triangle inequality to estimate
\[
\norm{u\circ Z_{\lambda}}_{\mathcal{X}_a}\leq\norm{u}_{\mathcal{X}_a}+a\norm{(D^2 u)\circ Z_{\lambda}-D^2(u\circ Z_{\lambda})}_{L^r(\Omega)}.
\]
Thanks to Lemma \ref{lem:map_d_compose}, we have 
\[
\norm{\Delta(u\circ Z_{\lambda})-(\Delta u)\circ Z_{\lambda}}_{L^r(\Omega)}\leq \norm{DZ_{\lambda}DZ_{\lambda}^T-I}_{L^{\infty}(\Omega)}\norm{D^2 u}_{L^r(\Omega)}+\norm{Du}_{L^{\infty}(\Omega)}\norm{\Delta Z_{\lambda}}_{L^r(\Omega)},
\]
and 
\[
\norm{(D^2 u)\circ Z_{\lambda}-D^2(u\circ Z_{\lambda})}_{L^r(\Omega)}\leq \norm{D^2 u}_{L^r(\Omega)}\sum_{i,j=1}^d\norm{\partial_i Z_{\lambda}\otimes \partial_j Z_{\lambda}-e_i\otimes e_j}_{L^{\infty}(\Omega)}+\norm{Du}_{L^{\infty}(\Omega)}\norm{D^2 Z_{\lambda}}_{L^r(\Omega)}
\]
Thus, 
\[
a\norm{\Delta(u\circ Z_{\lambda})-(\Delta u)\circ Z_{\lambda}}_{L^r(\Omega)}\leq \norm{u}_{\mathcal{X}_a}\Big(\norm{DZ_{\lambda}DZ_{\lambda}^T-I}_{L^{\infty}(\Omega)}+aK_a\norm{\Delta Z_{\lambda}}_{L^r(\Omega)}\Big),
\]
and
\[
a\norm{(D^2 u)\circ Z_{\lambda}-D^2(u\circ Z_{\lambda})}_{L^r(\Omega)}\leq \norm{ u}_{\mathcal{X}_a}\big(\sum_{i,j=1}^d\norm{\partial_i Z_{\lambda}\otimes \partial_j Z_{\lambda}-e_i\otimes e_j}_{L^{\infty}(\Omega)}+aK_a\norm{D^2 Z_{\lambda}}_{L^r(\Omega)}\big).
\]
From Corollary \ref{cor:lipschitz_estimate}, we have
\[
\norm{d\PP(Z_{\lambda};u\circ Z_{\lambda})}_{r,r}\leq 2\norm{\PP}_{r,r}^2\norm{Du}_{L^{\infty}(\Omega)}\leq 2K_a\norm{\PP}_{r,r}^2\norm{u}_{\mathcal{X}_a}
\]
From the definition of $u$ and $w$, and the bound (\ref{eqn:sobolev3}), we have 
\[
\norm{u}_{\mathcal{X}_a}\leq \bar{C}_r\norm{w}_{L^r(\Omega)}=\bar{C}_r.
\]
Thus, combining our work, we can conclude that
\[
-1+\norm{\PP}_{r,r}\bar{C}_r\Big(\norm{DZ_{\lambda}DZ_{\lambda}^T-I}_{L^{\infty}(\Omega)}+aK_a\norm{\Delta Z_{\lambda}}_{L^r(\Omega)}\Big)+2K_a\norm{\PP}_{r,r}^2\bar{C}_r\norm{(I-a\Delta)(Z_{\lambda}-S)}_{L^r(\Omega)}\geq 
\]
\[
-\bar{C}_r\frac{F(\tilde{Z})}{\lambda}\Big(1+\sum_{i,j=1}^d\norm{\partial_i Z_{\lambda}\otimes \partial_j Z_{\lambda}-e_i\otimes e_j}_{L^{\infty}(\Omega)}+aK_a\norm{D^2 Z_{\lambda}}_{L^r(\Omega)}\Big)
\]
The result now follows from the definition of $C_r$. 
\end{proof}

We conclude this subsection with an estimate for $K_a$.
\begin{lemma} \label{lem:B_and_K_estimates} 
If $K_a$ is defined as in Proposition \ref{prop:simplified_slope_bound}, 
then
\begin{equation}\label{K_a_bd}
K_a\lesssim_{d,r} a^{-\frac{d+r}{2r}}.
\end{equation}
\end{lemma}
\begin{proof}
By the Gagliardo-Nirenberg interpolation inequality, we have
\[
\norm{D f}_{L^{\infty}(\Omega)}\lesssim_{d,r}  \norm{D^2 f}_{L^r(\Omega)}^{\frac{d+r}{2r}}\norm{f}_{L^r(\Omega)}^{\frac{r-d}{2r}}\leq a^{-\frac{d+r}{2r}}\norm{f}_{\mathcal{X}_a}.
\]
Hence, $K_a\lesssim_{d,r} a^{-\frac{d+r}{2r}}$.
\end{proof}

\subsection{Construction of the starting point $\tilde{Z}$ and the proof of Theorem~\ref{thm:main_2}} 

In order to prove the failure of inequality (\ref{eq:specific_slope_ineq}), it remains to choose an appropriate starting point $\tilde{Z}$.

Let us define $u^*$ to be the solution to the Stokes resolvent problem
\begin{equation}\label{eq:u_star_def}
(I +a\mathcal A) u^*+\mathbb P(I-\Delta)(\id-S)=0, \quad \nabla \cdot u^*=0,\quad u^*|_{\partial\Omega}=0.
\end{equation}
We now apply Lemma \ref{stokes} to obtain that
\begin{equation}\label{eq:linearized_projectionII}
\norm{u^*}_{L^r(\Omega)}+ \norm{aD^2 u^*}_{L^r(\Omega)} \lesssim_{d, r} \norm{(I-a\Delta)(S-\id)}_{L^r(\Omega)} .
\end{equation}
Thus, 
\begin{equation}\label{eq:linearized_projectionII2}
\norm{(I-a\Delta)(S-\id-u^*)}_{L^r(\Omega)}  \lesssim_{d, r} \norm{(I-a\Delta)(S-\id)}_{L^r(\Omega)} .
\end{equation}

We shall now use this $u^*$ to construct the starting point $\tilde{Z}$. 
\begin{prop}\label{prop:reference_point}
Suppose that $Y:[0,\infty)\times\Omega \to \Omega$ is a map that satisfies 
\[
Y(t,x)=\id+\int_0^t u^*(Y(s,x))\, ds, \quad \det(DY(t,x))= 1,
\]
where $u^*$ is defined as in (\ref{eq:u_star_def}).
If we set $\tilde{Z}(x) :=Y(1,x)$ then 
\[
\norm{\tilde{Z}-\id}_{\mathcal{X}_a}\leq C_r\delta + \delta ^2K_a\Big(1+C_r\big(1+m_0( \delta C_rK_a)\big)\Big)
\]
and
\begin{equation}
F(\tilde{Z})\leq C_r\delta ^2K_a\Big(1+C_r\big(1+m_0( \delta C_rK_a)\big)\Big).
\end{equation}
Here $K_a$ and $C_r$ are defined as in Proposition \ref{prop:simplified_slope_bound}, $\delta:=\norm{((I-a\Delta)(S-\id)}_{L^r(\Omega)}$ and 
\[
m_0(t):=t \Big(1+  C_r\big(2 e^t +2t e^{2t}+ t^2 e^{3t} \big) \Big).
\]
Note that $K_a$ has an upper bound by \eqref{K_a_bd}. 
\end{prop}

\begin{proof}
We begin by noticing that $u^*$ must solve the equation 
\begin{equation}\label{eq:sep11.2020.1}
\PP(I-a\Delta)(S-\id-u^*)=0.
\end{equation}
By triangle inequality,
\[
\norm{\tilde{Z}-\id-u^*}_{\mathcal{X}_a}\leq \int_0^1 \norm{u^*\circ Y(s,\cdot)-u^*}_{\mathcal{X}_a}\, ds.
\]
By Lemma~\ref{lem:aug20.4} and Corollary~\ref{cor:lipschitz_estimate},
\[
F(\tilde{Z})\leq \norm{\PP(I-a\Delta)(\tilde{Z}-S)}_{L^r(\Omega)}+\norm{D\tilde{Z}-I}_{L^{\infty}(\Omega)}\norm{\tilde{Z}-S}_{\mathcal{X}_a}.
\]
We can then estimate
\[
\norm{\tilde{Z}-S}_{\mathcal{X}_a}\leq \norm{S-\id-u^*}_{\mathcal{X}_a}+\norm{\tilde{Z}-\id-u^*}_{\mathcal{X}_a}\leq
C_r\delta+\int_0^1 \norm{u^*\circ Y(s,\cdot)-u^*}_{\mathcal{X}_a}\, ds.
\]
By \eqref{eq:sep11.2020.1} we have 
\[
\PP(I-a\Delta)(\tilde{Z}-S)=\PP(I-a\Delta)(\tilde{Z}-\id-u^*),
\]
which gives us
\[
\norm{\PP(I-a\Delta)(\tilde{Z}-S)}_{L^r(\Omega)}\leq C_r\int_0^1 \norm{u^*\circ Y(s,\cdot)-u^*}_{\mathcal{X}_a}\, ds. 
\]

Focusing on the term $\int_0^1 \norm{u^*\circ Y(s,\cdot)-u^*}_{\mathcal{X}_a}\, ds $, we have the bound 
$$
\begin{array}{lll}
\int_0^1 \norm{u^*\circ Y(s,\cdot)-u^*}_{\mathcal{X}_a}\, ds&\leq& \int_0^1 \norm{Du^*}_{L^{\infty}(\Omega)}\norm{Y(s,\cdot)-\id}_{L^r(\Omega)}+a\norm{\Delta(u^*\circ Y(s,\cdot)-u^*)}_{L^r(\Omega)}\, ds\\ \\
&\leq& \frac{1}{2}\norm{Du^*}_{L^{\infty}(\Omega)}\norm{u^*}_{L^r(\Omega)}+\int_0^1 a\norm{\Delta(u^*\circ Y(s,\cdot)-u^*)}_{L^r(\Omega)}\, ds.
\end{array}
$$
A direct calculation gives the estimate
\[
\int_0^1 a\norm{\Delta(u^*\circ Y(s,\cdot)-u^*)}_{L^r(\Omega)}\, ds\leq
\]
\[
 \int_0^1 a\norm{\Delta Y(s,\cdot)}_{L^r(\Omega)}\norm{D u^*}_{L^{\infty}(\Omega)}+a\norm{D^2 u^*}_{L^r(\Omega)}\norm{DY(s,\cdot)-I}_{L^{\infty}(\Omega)}(1+\norm{DY(s,\cdot)-I}_{L^{\infty}(\Omega)})\, ds.
\]

It is straightforward to obtain the estimates
\[
\norm{DY(t,\cdot)-I}_{L^{\infty}(\Omega)}\leq \int_0^t \norm{Du^*}_{L^{\infty}(\Omega)}(1+\norm{DY(t,\cdot)-I}_{L^{\infty}(\Omega)})\, ds.
\]
and 
\[
\norm{\Delta Y(t,\cdot)}_{L^{r}(\Omega)}\leq \int_0^t \norm{Du^*}_{L^{\infty}(\Omega)}\norm{\Delta Y(s,\cdot)}_{L^r(\Omega)}+\norm{D^2 u^*}_{L^{r}(\Omega)}\norm{DY(s,\cdot)DY^T(s,\cdot)}_{L^{\infty}(\Omega)}.
\]
Hence, Gronwall's inequality gives us 
\[
\norm{DY(t,\cdot)-I}_{L^{\infty}(\Omega)}\leq t\norm{Du^*}_{L^{\infty}(\Omega)}\exp(t\norm{Du^*}_{L^{\infty}(\Omega)}),
\]
and 
\[
\norm{\Delta Y(t,\cdot)}_{L^{r}(\Omega)}\leq 
\norm{D^2 u^*}_{L^{r}(\Omega)}\sum_{j=0}^2 \big(t\norm{Du^*}_{L^{\infty}(\Omega)}\big)^j\exp((j+1)t\norm{Du^*}_{L^{\infty}(\Omega)}).
\]

By Lemma \ref{stokes} and the definition of $K_a$, we have 
\[
\norm{u^*}_{\mathcal{X}_a}\leq C_r\delta, \quad \norm{Du^*}_{L^{\infty}(\Omega)}\leq C_rK_a \delta.
\]
Putting together our computations we get
\begin{equation}\label{eq:u_lag_diff}
\int_0^1 \norm{u^*\circ Y(s,\cdot)-u^*}_{\mathcal{X}_a}\, ds\leq C_r\delta ^2K_a\Big(1+2C_re^{C_r\delta K_a}+2C_r\delta K_ae^{2C_r\delta K_a}+C_r(C_r\delta K_a)^2e^{3C_r\delta K_a}\Big),
\end{equation}
and our estimates of $F(\tilde{Z})$ and $\norm{\tilde{Z}-\id}_{\mathcal{X}_a}$ now follow. 
\end{proof}

Now we are ready to prove the existence of a point $Z\in \mathcal D{\rm iff}_{\id}(\Omega)$ such that the critical point equation (\ref{eq:ppcrit}) is satisfied.  We shall proceed by combining Propositions \ref{prop:simplified_slope_bound} and \ref{prop:reference_point}, and estimating the remaining quantities in terms of $a$ and $\delta$.  
\begin{thm}\label{critical_pt} 
If $\delta$ and  $a^{-\frac{d+r}{2r}}\delta$  are sufficiently small (depending on $d$ and $r$ only), then there exists a constant  $\lambda>0$  such that
\[
F(Z_{\lambda})=0 \quad \textrm{and} \quad \lambda \lesssim_{d,r} \delta^2 a^{-\frac{d+r}{2r}}.
\]
\end{thm}

\begin{proof}
For each $\lambda>0$, let $Z_{\lambda}$ be the point provided by Ekeland's variational principle starting from the point $\tilde{Z}$ constructed in Proposition \ref{prop:reference_point}.
If  $F(Z_{\lambda})\neq 0$, then Proposition \ref{prop:simplified_slope_bound} provides us with the ``slope inequality"
\[
-1+C_r^2\Big(\norm{DZ_{\lambda}DZ_{\lambda}^T-I}_{L^{\infty}(\Omega)}+aK_a\norm{\Delta Z_{\lambda}}_{L^r(\Omega)}\Big)+2K_aC_r^3\norm{(I-a\Delta)(Z_{\lambda}-S)}_{L^r(\Omega)}\geq 
\]
\[
-C_r\frac{F(\tilde{Z})}{\lambda}\Big(1+\sum_{i,j=1}^d\norm{\partial_i Z_{\lambda}\otimes \partial_j Z_{\lambda}-e_i\otimes e_j}_{L^{\infty}(\Omega)}+aK_a\norm{D^2 Z_{\lambda}}_{L^r(\Omega)}\Big)\]
Our goal is to rewrite this inequality in terms of $\delta$ and $a$ to derive a contradiction when $\delta$ and $a^{-\frac{d+r}{2r}}\delta$ are small enough.

\medskip

 Let us choose $\lambda:=\gamma_0F(\tilde{Z})$ for some constant $\gamma_0>0$, and set
   \[b_1:=\norm{Z_{\lambda}-\id}_{\mathcal{X}_a}.\]  
  Recall that  $\lambda$ is positive since $0<F(Z_{\lambda})\leq F(\tilde{Z})$.   We can then write
\[
\norm{(I-a\Delta)(Z_{\lambda}-S)}_{L^r(\Omega)}\leq \delta+\norm{Z_{\lambda}-\id}_{\mathcal{X}_a} = \delta + b_1,
\]
\[
\norm{DZ_{\lambda}DZ_{\lambda}^T-I}_{L^{\infty}(\Omega)}\leq K_a \norm{Z_{\lambda}-\id}_{\mathcal{X}_a}\big(2+K_a\norm{Z_{\lambda}-\id}_{\mathcal{X}_a}\big)=K_ab_1(2+K_ab_1),
\]
and 
\[
\sum_{i,j=1}^d\norm{\partial_i Z_{\lambda}\otimes \partial_j Z_{\lambda}-e_i\otimes e_j}_{L^{\infty}(\Omega)}\leq d^2K_a \norm{Z_{\lambda}-\id}_{\mathcal{X}_a}\big(2+K_a\norm{Z_{\lambda}-\id}_{\mathcal{X}_a}\big)=d^2 K_ab_1(2+K_ab_1).
\]
We also note that 
\[
a\norm{ D^2 Z_{\lambda}}_{L^r(\Omega)}=aK_a\norm{D^2 (Z_{\lambda}-\id)}_{L^r(\Omega)}\leq K_ab_1,
\]
and hence
\[
a\norm{ \Delta Z_{\lambda}}_{L^r(\Omega)}\leq K_a b_1
\]
Using these bounds, the  slope inequality can now be written as 
\[
-1+C_r^2\Big(K_ab_1(2+K_ab_1)+K_ab_1\Big)+2K_aC_r^3(\delta + b_1)\geq 
-\frac{C_r}{\gamma_0}\Big(1+d^2 K_ab_1(2+K_ab_1)+K_ab_1\Big)
\]
Dropping constants and rearranging, we have shown that
\begin{equation}\label{inequality_0}
1-\gamma_0^{-1}\lesssim K_a(\delta + b_1)+K_ab_1(2+K_ab_1)+K_ab_1
+\gamma_0^{-1}\Big( K_ab_1(2+K_ab_1)+K_ab_1\Big).
\end{equation}

Let us now define 
\begin{equation*}
b_2:=K_ab_1.
\end{equation*} 
Using the above calculations, we can rewrite \eqref{inequality_0}  as
\begin{equation}\label{eq:simplified_epsilon_slope_ineq}
1-\gamma_0^{-1} \lesssim_r  K_a\delta +b_2^2 + b_2 +\gamma_0^{-1}(b_2+b_2^2) 
\end{equation}

Since $\norm{Z_{\lambda}-\tilde{Z}}_{\mathcal{X}_a} \leq \lambda$, as long as $a^{-\frac{d+r}{2r}}\delta$ are sufficiently small,  Proposition \ref{prop:reference_point} yields
\[
b_1=\norm{Z_{\lambda}-\id}_{\mathcal{X}_a}\leq \lambda+\norm{\tilde{Z}-\id}_{\mathcal{X}_a} \lesssim_r  \lambda+ \delta + K_a\delta^2.
\]

Moreover, Proposition~\ref{prop:reference_point} and Lemma \ref{lem:B_and_K_estimates} yields that 
$$
K_a \lesssim_{d,r} a^{-\frac{d+r}{2r}} \quad \hbox{ and }\quad \lambda= \gamma_0 F(\tilde{Z})\lesssim_{d,r}  \gamma_0 K_a \delta^2 \lesssim_{d,r} \gamma_0a^{-\frac{d+r}{2r}}\delta^2.
$$

 Hence we see that  $K_ab_1\lesssim_{d,r} a^{-\frac{d+r}{2r}}\delta+(a^{-\frac{d+r}{2r}}\delta)^2$.

  Thus, when $a^{-\frac{d+r}{2r}}\delta$ is sufficiently small, \eqref{eq:simplified_epsilon_slope_ineq} fails for $\gamma_0 = 2$. Hence, we can deduce by contradiction that $F(Z_{\lambda})=0$ when $\lambda=\gamma_0 F(\tilde{Z})$.  The conclusion of the theorem follows from the bounds on $F(\tilde{Z})$ and $K_a$.
  
  \end{proof}

\textbf{Proof of Theorem~\ref{thm:main_2}}. 

When $\delta$, $a^{-\frac{d+r}{2r}}\delta$ and $a^{-\frac{d+3r}{2r}}\delta^2$ are sufficiently small, Theorem~\ref{critical_pt} yields a point $Z_{\lambda}$ such that
\begin{equation}\label{eq:100}
F(Z_\lambda)=0 \hbox{ and } \norm{Z_{\lambda}-\tilde{Z}}_{\mathcal{X}_a} \leq \lambda =\gamma\delta^2 a^{-\frac{d+r}{2r}},
\end{equation}
 where $\tilde{Z}$ is as given in Proposition \ref{prop:reference_point}.   Since $F(Z_{\lambda})=0$ and $Z_{\lambda}\in W^{2,r}(\Omega)$ we know that
\[
(I-a\Delta)(Z_{\lambda}-S)+D^T(\cof(DZ)q^*)=0,
\]
for some function $q^*\in W^{1,r}(\Omega)$.   Theorem~\ref{thm:main_2} will now follow from Theorem~\ref{thm:unique} if we can show that $\norm{q^*}_{L^{\infty}(\Omega)}\leq a\sigma^2(2+3(1+\sqrt{3}))^{-1}$, where $\sigma$ is the smallest singular value of $DZ_{\lambda}$.

Note that 
\[
\norm{DZ_{\lambda}-I}_{L^{\infty}(\Omega)}\leq K_a(\lambda+\norm{\tilde{Z}-\id}_{\mathcal{X}_a})\lesssim_{d,r}  a^{-\frac{d+r}{2r}}\delta,
\]
Thus, 
\begin{equation}\label{sigma}
\sigma\geq 1-\theta_{d,r} a^{-\frac{d+r}{2r}}\delta
\end{equation}
for some constant $\theta_{d,r}>0$. 

Define $p^*:=q^*\circ Z_{\lambda}^{-1}$ and note that $\norm{q^*}_{L^{\infty}(\Omega)}=\norm{p^*}_{L^{\infty}(\Omega)}$.  We have $D^T(\cof(DZ)\nabla q)=\nabla p^*(Z_{\lambda})$.
  Since $Z_{\lambda}$ is measure preserving, we have
\[
\norm{\nabla p^*}_{L^r(\Omega)}=\norm{(I-a\Delta)(Z_{\lambda}-S)}_{L^r(\Omega)}\leq \lambda+\norm{(I-a\Delta)(\tilde{Z}-S)}_{L^r(\Omega)}.
\]
Recalling the definition of $u^*$ from (\ref{eq:u_star_def})
and 
\[
\delta':=\norm{S-\id-u^*}_{\mathcal{X}_a}
\]
we have 
\[
\norm{(I-a\Delta)(\tilde{Z}-S)}_{L^r(\Omega)}\leq \delta'+\norm{(I-a\Delta)(\tilde{Z}-\id-u^*)}_{L^r(\Omega)}.
\]
Since $\norm{(I-a\Delta)(\tilde{Z}-\id-u^*)}_{L^r(\Omega)}$ is bounded by the quantity on the left hand side of (\ref{eq:u_lag_diff}), we can conclude that
\[
\norm{(I-a\Delta)(\tilde{Z}-S)}_{L^r(\Omega)}\lesssim \delta'+a^{-\frac{d+r}{2r}}\delta^2,
\]
when $\delta$ and $a^{-\frac{d+r}{2r}}\delta$ are sufficiently small. In particular from \eqref{eq:100} it follows that 
$$
\norm{(I-a\Delta)(Z_{\lambda}-S)}_{L^r(\Omega)} \leq \delta'+a^{-\frac{d+r}{2r}}\delta^2.
$$

It is now clear from \eqref{sigma} and  the Poincar\'{e}-Wirtinger inequality  that 
\[
 \Big\| p^* - {1\over |\Omega|} \int_\Omega p^*(x) dx \Big\|_{L^{\infty}(\Omega)} \leq a\sigma^2 \big(2+3(1+\sqrt{3})\big)^{-1}
 \]
will hold as long as $a^{-1}(\delta'+a^{-\frac{d+r}{2r}}\delta^2)$ is sufficiently small. 
 
 \hfill$\Box$

\section{Application to Navier-Stokes equation}\label{section:6}

In this section we prove Theorem~\ref{thm:NSE}. For a given initial data $v_0\in L^r(\Omega; \Rd)$ with $r>d$, we will construct discrete-time solutions that generates the unique mild solution of the Navier-Stokes equations \eqref{eq:NSE}  as well as the associated Lagrangian flow \eqref{eq:lagrangian_ns}.

\subsection{The discrete scheme: Lagrangian and Eulerian viewpoint} Let $\mathcal{A}$ denote the Stokes operator introduced in \eqref{stokes_op}. We begin by recalling the discrete-in-time scheme to construct Lagrangian and Eulerian solutions to the Navier-Stokes equations.  Given an initial velocity $v_0$, we set $v_{0,\tau}=v_0$ and iterate the following steps:
\begin{equation}\label{eq:ns_scheme_step_1}
(I-\mt\Delta)S_{n+1,\tau}=\id+\tau v_{n,\tau}, \quad S_{n+1,\tau}|_{\partial\Omega}=\id,
\end{equation}
  \begin{equation}\label{eq:ns_scheme_step_1.5}
Z_{n+1,\tau}\in \argmin_{Z\in \mathcal{D}\textrm{iff}_{\id}(\Omega)} \frac{1}{2}\norm{Z-S_{n+1,\tau}}_{L^2(\Omega)}^2+\frac{\mt}{2}\norm{DZ-DS_{n+1}}_{L^2(\Omega)}^2,
\end{equation}
\begin{equation}\label{eq:ns_scheme_step_2}
w_{n+1,\tau}:=  Z_{n+1\, \#}v_{n,\tau},
\end{equation}
\begin{equation}\label{eq:ns_scheme_step_3}
v_{n+1,\tau}:= e^{-\mt \mathcal{A}}w_{n+1}.
\end{equation}
Note that $v_{n,\tau}$ has zero trace on $\partial\Omega$ from the definition.

\medskip

Due to Theorem~\ref{thm:main_2}, $Z_{n+1}$ exists as long as $\norm{v_n}_{L^r(\Omega)}$ is bounded (see Lemma~\ref{from_sec_5}), which makes the scheme well-defined.
We shall also use the scheme to construct discrete-in-time Lagrangian solutions $X_{n,\tau}$ by setting $X_{0,\tau}=\id$ and iterating
\begin{equation}\label{eq:nslagrangian}
X_{n+1,\tau}=Z_{n+1,\tau}\circ X_{n,\tau}.
\end{equation}

  Note that the steps (\ref{eq:ns_scheme_step_2}) and  (\ref{eq:ns_scheme_step_3}) can be understood as a splitting scheme for Navier-Stokes.  Step  (\ref{eq:ns_scheme_step_2}) accounts for the transportation of the velocity field, while step (\ref{eq:ns_scheme_step_3}) accounts for the linear parts of the equation.  The non-standard aspect of this scheme is that we advect the vector field with the projection map $Z_{n+1,\tau}$.   This lends a great deal of stability, since it makes the scheme much more implicit.   Furthermore, $Z_{n+1,\tau}$ is measure preserving, thus, we will see that the scheme automatically satisfies a discrete version of the energy dissipation inequality and the Navier-Stokes Duhamel formula (see Lemmas \ref{lem:edi} and \ref{duhamel}). 

\medskip

Our ultimate aim is to show that the velocity iterates $v_{n+1,\tau}$ and the Lagrangian maps $X_{n+1,\tau}$ converge to Eulerian and Lagrangian solutions of Navier-Stokes respectively as $\tau$ tends to zero.  To that end we will introduce piecewise constant interpolations $v_{\tau}, Z_{\tau}, X_{\tau}, \tilde{v}_{\tau}$ defined as follows: for $U$ denoting $v, Z, X$ and $\tilde{v}$,
\begin{equation}
U_{\tau}(t,x):= U_{n+1,\tau}(x) \quad \hbox{ if }  n\tau \leq t <(n+1)\tau.
\end{equation}

Now we are ready to analyze the scheme.  Let us begin by translating the estimates from Section 5 into our current setting. The following statement is a direct consequence of Theorem~\ref{thm:main_2}. Note that here we have 
$$
a=\mt, \,\,  \delta=\tau\norm{v_{N-1,\tau}}_{L^r(\Omega)} \hbox{ and }\delta'=0.
$$ 

\begin{lemma}\label{from_sec_5}
There exists a constant $C>0$  only depending on $r$ such that the following holds. 
Suppose that (\ref{eq:ns_scheme_step_1}-\ref{eq:ns_scheme_step_3}) are well-defined for $0\leq n \leq N-1$ and suppose that $v_{N-1,\tau}$ satisfies 
\begin{equation}\label{small}
\norm{v_{N-1,\tau}}_{L^r(\Omega)} \leq C \mu^{d+3r \over 4r} \tau^{d-r \over 4r}
\end{equation}
for some constant $C>0$. If $C$ is sufficiently small then $v_{N,\tau}$ is well-defined by the scheme.

\end{lemma}

Since $r>d$, Lemma~\ref{from_sec_5} will follow if we can show that $\norm{v_{n,\tau}}_{L^r(\Omega)}$ is uniformly bounded with respect to $\tau$. This is what we will show  in Section 6.3 for a finite time period $0\leq n \leq [\frac{T^*}{\tau}]$. To this end we first present a preliminary estimate that connects $\tilde{v}$ and $v$.

\begin{prop}\label{prop:scheme_ekeland_estimates}
Given some $\tau>0$, let $\{v_{n,\tau}, Z_{n,\tau}, \tilde{v}_{n,\tau}\}_{n\geq 0}$ be the sequence of iterates generated by (\ref{eq:ns_scheme_step_1}-\ref{eq:ns_scheme_step_3}) and \eqref{eq:nslagrangian}.
If $v_{n,\tau}$ satisfies the $L^r$ norm bound from Lemma \ref{from_sec_5}, then
\begin{equation}\label{eq:sep20.2020.0}
\norm{\tilde{v}_{n+1,\tau}}_{L^r(\Omega)}\lesssim_{d,r}\norm{v_{n,\tau}}_{L^r(\Omega)}+\mu^{-\frac{d+r}{2r}} \tau^{1-\frac{d+r}{2r}}\norm{v_{n,\tau}}_{L^r(\Omega)}^2.
\end{equation}
Furthermore, if $v_{n,\tau}\in H^1(\Omega)$ then 
\begin{equation}
\norm{\tilde{v}_{n+1,\tau}-v_{n,\tau}}_{L^2(\Omega)}\lesssim_{d,r} (\mt)^{1/2}\norm{Dv_{n,\tau}}_{L^2(\Omega)}+\mu^{-\frac{d+r}{2r}} \tau^{1-\frac{d+r}{2r}}\norm{v_{n,\tau}}_{L^r(\Omega)}^2.
\end{equation}
\end{prop}

\begin{proof}
Let $\tilde{Z}_{n+1,\tau}$ be the reference point from Proposition \ref{prop:reference_point} constructed from the map $S_{n+1,\tau}$.   By the triangle inequality, we have 
\[
\norm{\tilde{v}_{n+1,\tau}}_{L^r(\Omega)}=\norm{\frac{Z_{n+1}-\id}{\tau}}_{L^r(\Omega)}\leq \norm{\frac{Z_{n+1}-\tilde{Z}_{n+1,\tau}}{\tau}}_{L^r(\Omega)}+\norm{\frac{\tilde{Z}_{n+1}-\id}{\tau}}_{L^r(\Omega)}.
\]
Clearly, 
\[
\norm{\frac{Z_{n+1}-\tilde{Z}_{n+1,\tau}}{\tau}}_{L^r(\Omega)}\leq \norm{\frac{Z_{n+1}-\tilde{Z}_{n+1,\tau}}{\tau}}_{\mathcal{X}_{\mt}}.
\]
Thanks to Theorem \ref{critical_pt}, we must have 
\[
\norm{\frac{Z_{n+1}-\tilde{Z}_{n+1,\tau}}{\tau}}_{\mathcal{X}_{\mt}}\lesssim_{d,r} \mu^{-\frac{d+r}{2r}} \tau^{1-\frac{d+r}{2r}}\norm{v_{n,\tau}}_{L^r(\Omega)}^2.
\]
If we let $u_{n+1,\tau}^*$ denote the solution to the Stokes resolvent problem 
\[
(I+\mt\mathcal{A}) u_{n+1,\tau}^*=v_{n,\tau},  \quad \nabla \cdot u_{n+1,\tau}^*=0, \quad u_{n+1,\tau}^*|_{\partial\Omega}=0,
\]
then it is clear from the reference point construction in Proposition \ref{prop:reference_point} that
\[
\norm{\frac{\tilde{Z}_{n+1}-\id}{\tau}}_{L^r(\Omega)}\leq \norm{u_{n+1,\tau}^*}_{L^r(\Omega)}\lesssim_{d,r} \norm{v_{n,\tau}}_{L^r(\Omega)},
\]
where the last inequality follows from Lemma \ref{stokes}. 
Therefore, 
\[
\norm{\tilde{v}_{n+1,\tau}}_{L^r(\Omega)}\lesssim_{d,r}
 \mu^{-\frac{d+r}{2r}} \tau^{1-\frac{d+r}{2r}}\norm{v_{n,\tau}}_{L^r(\Omega)}^2+\norm{v_{n,\tau}}_{L^r(\Omega)}
\]

Now we turn to the second statement.
Following a similar idea to the above, we can estimate
\[
\norm{\tilde{v}_{n+1,\tau}-v_{n,\tau}}_{L^2(\Omega)}\leq \norm{\frac{Z_{n+1,\tau}-\tilde{Z}_{n+1,\tau}}{\tau}}_{L^2(\Omega)}+\norm{\frac{\tilde{Z}_{n+1,\tau}-\id}{\tau}-u^*_{n+1,\tau}}_{L^2(\Omega)}+\norm{u^*_{n+1,\tau}-v_{n,\tau}}_{L^2(\Omega)}.
\]
Dominating $L^2$ by $\mathcal{X}_{\mt}$ we can estimate
\[
\norm{\frac{Z_{n+1,\tau}-\tilde{Z}_{n+1,\tau}}{\tau}}_{L^2(\Omega)}\lesssim \norm{\frac{Z_{n+1}-\tilde{Z}_{n+1,\tau}}{\tau}}_{\mathcal{X}_{\mt}}\lesssim_{d,r} \mu^{-\frac{d+r}{2r}} \tau^{1-\frac{d+r}{2r}}\norm{v_{n,\tau}}_{L^r(\Omega)}^2.
\]
Again, from the construction of the reference point, it is immediate that
\[
\norm{\frac{\tilde{Z}_{n+1,\tau}-\id}{\tau}-u^*_{n+1,\tau}}_{L^2(\Omega)}\, ds\leq \tau\norm{Du^*_{n,\tau}}_{L^{\infty}(\Omega)}\norm{u^*_{n,\tau}}_{L^2(\Omega)}.
\]
The Sobolev inequalities and Lemma \ref{stokes}, then give
\[
\tau\norm{Du^*_{n,\tau}}_{L^{\infty}(\Omega)}\norm{u^*_{n,\tau}}_{L^2(\Omega)}\lesssim \mu^{-\frac{d+r}{2r}}\tau^{1-\frac{d+r}{2r}}\norm{v_{n,\tau}}_{L^r(\Omega)}^2.
\]
Finally, we can use the equation satisfied by $u_{n+1,\tau}^*$ to compute
\[
\norm{u^*_{n+1,\tau}-v_{n,\tau}}_{L^2(\Omega)}^2= \mt \big( \mathcal{A} u^*_{n+1,\tau}, u^*_{n+1,\tau}-v_{n,\tau})= \mt \big( D u^*_{n+1,\tau}, Dv_{n,\tau})- \mt \norm{D u^*_{n+1,\tau}}_{L^2(\Omega)}^2 \leq \frac{\mt}{2}\norm{D v_{n,\tau}}^2_{L^2(\Omega)}.
\]
Combining these estimates we get the second result.

\end{proof}

Note that the bounds obtained in Proposition \ref{prop:scheme_ekeland_estimates} present a superlinear growth in $\tau$, and thus they cannot be iterated to generate a uniform bound.  This is because the one-step estimates do not take into account the regularizing effect of the viscosity in the Navier-Stokes equation. In the next subsection we will utilize an approximate Duhamel's fomula to obtain an improved estimate that leverages the regularization effect over time (see Lemma~\ref{lem:duhamel_velocity_bound}).

\subsection{Energy dissipation and Duhamel's formula} 

In this subsection, we will establish discrete analogues of the well-known Navier-Stokes energy dissipation inequality and Duhamel formula. In the following two lemmas, we assume that the scheme (\ref{eq:ns_scheme_step_1}-\ref{eq:ns_scheme_step_3}) is well defined for all iterates $1\leq n < N_{\tau}$. Note that $N_{\tau} \geq 1$ due to Lemma~\ref{from_sec_5}.

\begin{lemma}[Approximate energy dissipation inequality]\label{lem:edi}
\begin{equation}
\norm{v_{n+1,\tau}}_{L^2(\Omega)}^2+2\mt\norm{D v_{n+1,\tau}}_{L^2(\Omega)}^2\leq \norm{v_{n,\tau}}_{L^2(\Omega)}^2 \qquad\hbox{ for } 1\leq n <N_{\tau}
\end{equation}
and 
\begin{equation}\label{EDI}
\norm{v_{n+1,\tau}}_{L^2(\Omega)}^2+2\mt\sum_{j=m+1}^{n+1}\norm{Dv_{j,\tau}}_{L^2(\Omega)}^2\leq \norm{ v_{m,\tau}}_{L^2(\Omega)}^2 \quad \hbox{ for  } 1\leq  m\leq n<N_{\tau}.
\end{equation}
\end{lemma}
\begin{proof}
Let $h(t):=\norm{e^{-\mu t\mathcal{A}} \PP w_{n+1,\tau}}_{L^2(\Omega)}^2$.  Differentiating in time and then integrating by parts, we have $h'(t)=-2\mu \norm{D e^{-\mu t \mathcal{A}} \PP w_{n+1}}_{L^2(\Omega)}^2$.  Therefore,
\[
h(\tau)+\int_0^{\tau} 2\mu \norm{D e^{-\mu t\mathcal{A}}\PP w_{n+1}}_{L^2(\Omega)}^2\, dt\leq h(0).
\]
Integrating the time derivative of the non-stationary stokes equation against itself, one also has 
\[
2\mt \norm{D v_{n+1,\tau}}_{L^2(\Omega)}^2\leq 2\mu\int_{0}^{\tau} \norm{D e^{-s\mathcal{A}}\PP w_{n+1}}_{L^2(\Omega)}^2\, ds.
\] 
Thus, we can conclude that 
\[
\norm{v_{n+1,\tau}}_{L^2(\Omega)}^2+2\mt\norm{D v_{n+1,\tau}}_{L^2(\Omega)}^2\leq \norm{\PP w_{n+1,\tau}}_{L^2(\Omega)}^2\leq \norm{w_{n+1,\tau}}_{L^2(\Omega)}^2
\]

Finally, using the definition of $w_{n+1,\tau}$ we have 
\[
 \norm{ w_{n+1}}_{L^2(\Omega)}= \norm{ Z_{n+1\,\#}v_{n,\tau}}_{L^2(\Omega)}\leq \norm{v_{n,\tau}}_{L^2(\Omega)}.
\]
Combining our work, we obtain the first result.  The second result follows from iteration. 
\end{proof}

Next, we will show that our scheme satisfies a discrete analogue of the Navier-Stokes Duhamel formula.  The Duhamel formula will play the central role in our subsequent analysis. Indeed, we will characterize our solution using the Duhamel formula and we will also use it to obtain a short time bound on the velocity $L^r$ norm. 
\begin{lemma}[Approximate Duhamel formula]\label{lem:approximate_duhamel}  
 If $f:\Omega\to \RR^d$ is a smooth divergence free vector field that vanishes on $\partial\Omega$, then for any $n<N_{\tau}$ we have
\begin{equation}\label{duhamel}
(v_{n+1,\tau},f)=(v_0, f_{n+1})+\sum_{k=0}^n \Big(v_{k,\tau}, f_{n+1-k}\circ Z_{k+1,\tau}-f_{n+1-k}\Big),
\end{equation}
where $f_k=e^{ -\tau k\mu A} f$. 
\end{lemma}
\begin{proof}  \medskip

Let $f:\Omega \rightarrow \Rd$ be a smooth divergence free vector field that vanishes on $\partial \Omega$ and set $f_1:=e^{-\tau \mu \mathcal A} f.$ From integration by parts  and   using the definition of $w_{n+1,\tau}$, we see that
\[
(v_{n+1,\tau}, f)=\big(w_{n+1,\tau}, f_1\big)= (v_{n,\tau}, f_1\circ Z_{n+1,\tau}).
\]
Thus, we can conclude that
\begin{equation}\label{duhamel-discrete}
(v_{n+1,\tau},f)=(v_{n,\tau}, f_{1})+ \Big(v_{n,\tau}, f_{1}\circ Z_{n+1,\tau}-f_{1}\Big).
\end{equation}
Iterating this argument, we get the above result. 
\end{proof}

\subsection{$L^r$ norm control}
 
Based on the Duhamel formula,  we will now show that there exists a time $T^*>0$ such that $\norm{v_{n,\tau}}_{L^r(\Omega)}$ is bounded independently of $\tau$ for all $0\leq n\leq N_{\tau}:=\lfloor\frac{T^*}{\tau}\rfloor$.  This will establish \eqref{small} for all iterates in the range $0\leq n\leq N_{\tau}:=\lfloor\frac{T^*}{\tau}\rfloor$ once $\tau$ is sufficiently small.

\begin{lemma}\label{lem:duhamel_velocity_bound} 
\begin{align}
\norm{v_{n+1,\tau}}_{L^r(\Omega)}\lesssim_{r,d} & \;\; \norm{e^{-\tau(n+1)\mathcal{A}}v_0}_{L^r(\Omega)}\nonumber\\
+&\tau  \sum_{k=0}^n \big( \mu \tau(n+1-k)\big)^{-\frac{d+r}{2r}}\Big(\norm{v_{k,\tau}}_{L^r(\Omega)}^2+\mu^{-\frac{d+r}{2r}} \tau^{1-\frac{d+r}{2r}}\norm{v_{n,\tau}}_{L^r(\Omega)}^3\Big). \label{eq:may25.2020.8}
\end{align} 
\end{lemma} 


\begin{proof} 
Note that $v_{n+1,\tau}$ is well defined by Lemma~\ref{from_sec_5}.
Let $f$ be a smooth divergence free vector field vanishing on $\partial\Omega$, and define $f_k:=e^{ -\tau k\mu \mathcal{A}} f$ as in Lemma \ref{lem:approximate_duhamel}. We first use H\"older's inequality and second use Remark \ref{rem:w1p_lip} to conclude that 
\[
 \Big(v_{k,\tau}, f_{n+1-k}\circ Z_{k+1,\tau}-f_{n+1-k}\Big) \leq C\norm{v_{k,\tau}}_{L^r(\Omega)}  \| Z_{k+1,\tau}-\id\|_{L^{r}(\Omega)}\|Df_{n+1-k}\|_{L^{r \over r-2}(\Omega)}.
\]
Taking the supremum over smooth divergence free vector fields $f$ in the unit ball of $L^{\frac{r}{r-1}}(\Omega)$ and using Lemma \ref{lem:para_est_0}, we can conclude that 
\[
\sup_{\nabla \cdot f=0,\; \norm{f}_{L^{\frac{r}{r-1}}(\Omega)}\leq 1} (v_{n+1,\tau},f)\lesssim
\]
\[
 \norm{e^{-\tau(n+1)\mathcal{A}}v_0}_{L^r(\Omega)}
+\tau C \sum_{k=0}^n\big(  \mu\tau(n+1-k)\big)^{-\frac{d+r}{2r}}\norm{v_{k,\tau}}_{L^r(\Omega)}\norm{\frac{Z_{k+1,\tau}-\id}{\tau}}_{L^r(\Omega)},
\]
where the first inequality is from \eqref{duhamel}.
Note that 
\[
\norm{v_{n+1,\tau}}_{L^r(\Omega)}=\sup_{\norm{f}_{L^{\frac{r}{r-1}}(\Omega)}\leq 1} (v_{n+1,\tau},f)=\sup_{\norm{f}_{L^{\frac{r}{r-1}}(\Omega)}\leq 1} (v_{n+1,\tau},\PP f),
\]
where the last equality follows from the fact that $v_{n+1,\tau}$ is divergence free.   Hence,
\[
\norm{v_{n+1,\tau}}_{L^r(\Omega)} \leq \norm{\PP}_r \sup_{\nabla \cdot f=0,\; \norm{f}_{L^{\frac{r}{r-1}}(\Omega)}\leq 1} (v_{n+1,\tau},f).
\]
Applying Proposition \ref{prop:scheme_ekeland_estimates} to $\norm{\frac{Z_{k+1,\tau}-\id}{\tau}}_{L^r(\Omega)}$ we obtain the result.

\end{proof}

\begin{prop}\label{prop:critical_time}
There exists a time $T^*>0$ and some $\tau_0>0$ depending on $\norm{v_0}_{L^r(\Omega)}$, $r,d$ and $\mu$ such that for all $0<\tau \leq \tau_0$ if $\tau(n+1)< T^*$, then $\sup_{k\leq n+1} \norm{v_{k,\tau}}_{L^r(\Omega)}$ is bounded independently of $\tau$. 
\end{prop}

\begin{proof}
From Lemma \ref{lem:para_est_0} and Lemma  \ref{lem:duhamel_velocity_bound} there exist constants $C_0=C_0(r,d,\mu)$  
\[
\norm{v_{n+1,\tau}}_{L^r(\Omega)}\leq C_0\norm{v_0}_{L^r(\Omega)}+\tau  \sum_{k=0}^n \big( \mu \tau(n+1-k)\big)^{-\alpha}\Big(\norm{v_{k,\tau}}_{L^r(\Omega)}^2+\mu^{-\alpha} \tau^{1-\alpha}\norm{v_{n,\tau}}_{L^r(\Omega)}^3\Big). 
\]
where $\alpha := \frac{1}{2} + \frac{d}{2r} <1$. Hence if $\{h_k\}_{k\geq 0}$ solves $h_0=C_0\norm{v_0}_{L^r(\Omega)}$ and 
\begin{equation}\label{eq:h_sequence}
h_{n+1}=h_0+\tau C_1\sum_{k=0}^n \big( \mu \tau(n+1-k)\big)^{-\alpha}(h_k^2+\mu^{-\alpha} \tau^{1-\alpha}h_k^3),
\end{equation}
then $\norm{v_{k,\tau}}_{L^r(\Omega)}\leq h_{k}$.    Suppose that $h_k \leq M:= 2h_0$ for $k=1, \ldots, n$. Then we have 
$$
h_{n+1} \leq \frac{M}{2} + C_1(\tau \mu(n+1))^{1-\alpha} (M^2 + \mu^{-\alpha}\tau^{1-\alpha} M^3).
$$
This is less than $M$ as long as
$$
\tau(n+1)\leq T_D := (\frac{1}{4C_1M})^{\frac{1}{1-\alpha}} \hbox{ and } \tau \leq \tau_0:= (\frac{\mu^{\alpha}}{M})^{\frac{1}{1-\alpha}}.
$$
Thus, we see that the velocity doubling time $T_D$ is uniformly  bounded from below for all $\tau \in [0,\tau_0]$.  Since $T_D$ is a strictly positive lower bound for $T^*$, we are done. 
\end{proof}

\subsection{Convergence of the scheme to Eulerian and Lagrangian solutions}

Given the existence of the critical time $T^*>0$ from the previous section, we at last show that the approximate solutions  $v_{\tau}$ and $X_{\tau}$ converge to Eulerian and Lagrangian solutions of the Navier-Stokes equations respectively. 

\begin{prop}\label{prop:precompactness}
Let $T^*$ be as given in Proposition~\ref{prop:critical_time}. Then for any $T<T^*$ the family $\{v_{\tau}\}_{\tau\geq 0}$ is  uniformly bounded in 
\[
L^2\Big([0,T]; H^1_0(\Omega)\Big)\cap L^{\infty}\Big([0,T];L^2(\Omega)\Big)
\]
 and precompact in $L^2([0,T]\times\Omega)$.  
\end{prop}
\begin{proof}
The uniform boundedness is an immediate consequence of \eqref{EDI}.  Precompactness in $L^2([0,T]\times\Omega)$  will follow from modifications of the Aubin-Lions Lemma in \cite{guo_aubin_lions} if we can show that the discrete time derivatives $\frac{v_{n+1,\tau}-v_{n,\tau}}{\tau}$ are uniformly bounded in some weak space.  

Let $\mathcal{Y}:=\{g\in W^{2,r}_0(\Omega): \nabla \cdot g=0\}$.
We wish to estimate
\[
\max_{0\leq n<N} \Bignorm{\frac{v_{n+1,\tau}-v_{n,\tau}}{\tau}}_{\mathcal{Y}^*}. 
\]
Given $f\in\mathcal{Y}$, \eqref{duhamel-discrete} gives us
\[
\bigg(\frac{v_{n+1,\tau}-v_{n,\tau}}{\tau},f\bigg)=\bigg(v_{n,\tau}, \frac{f_1-f}{\tau}\bigg)+\bigg(v_{n,\tau}, \frac{f_1\circ Z_{n+1,\tau}-f_1}{\tau}\bigg)
\]
where $f_1=e^{\mt \mathcal{A}} f$. Thus, it is clear that 
\begin{align*}
\bigg(\frac{v_{n+1,\tau}-v_{n,\tau}}{\tau},f\bigg)\lesssim_{d,r}  & \norm{v_{n,\tau}}_{L^2(\Omega)}\norm{A f}_{L^2(\Omega)}+\norm{v_{n,\tau}}_{L^2(\Omega)}\norm{\tilde{v}_{n+1,\tau}}_{L^2(\Omega)}\norm{Df_1}_{L^{\infty}(\Omega)} \\ 
\lesssim_{d,r} & \norm{f}_{\mathcal{Y}}\bigg(\norm{v_n}_{L^2(\Omega)}+\frac{3}{2}\norm{v_{n,\tau}}_{L^2(\Omega)}^2+\frac{1}{2}\norm{v_{n,\tau}-\tilde{v}_{n+1,\tau}}_{L^2(\Omega)}^2\bigg),
\end{align*}
where the last inequality follows from the Sobolev inequalities.
Note that 
\[
\frac{1}{2}\norm{v_{n,\tau}-\tilde{v}_{n+1,\tau}}_{L^2(\Omega)}^2\leq J_{\mu \tau}(Z_{n+1},v_{n,\tau})+\frac{1}{2}\norm{v_{n,\tau}}^2_{L^2(\Omega)}\leq J_{\mu \tau}(\id, v_{n,\tau})+\frac{1}{2}\norm{v_{n,\tau}}_{L^2(\Omega)}^2=\frac{1}{2}\norm{v_{n,\tau}}_{L^2(\Omega)}^2.
\]
Now we can conclude that 
\[
\max_{0\leq n\leq N}  \Bignorm{\frac{v_{n+1,\tau}-v_{n,\tau}}{\tau}}_{\mathcal{Y}^*}\leq \sup_n \norm{v_{n,\tau}}_{L^2(\Omega)}+2\norm{v_{n,\tau}}_{L^2(\Omega)}^2\leq \norm{v_{0}}_{L^2(\Omega)}+2\norm{v_0}_{L^2(\Omega)}^2,
\]
where the final inequality follows from \eqref{EDI}. 
\end{proof}

\begin{thm}\label{thm:conv_1}
Let $T^*$ be as given in Proposition \ref{prop:critical_time}. Then for any $T<T^*$ there exists $v\in L^{\infty}([0,T];L^r(\Omega))\cap L^2([0,T];H^1_0(\Omega))$ such that
 \[
 \lim_{\tau\to 0} \norm{v-v_{\tau}}_{L^2([0,T]\times\Omega)}=0
 \]
and
 \begin{equation}\label{eq:duhamel_solution}
 v(t,x)= e^{-\mu t\mathcal{A}}v_0 (x)- \int_0^t e^{-\mu (t-s)\mathcal{A}} \PP \nabla \cdot \big( v(s,x)\otimes v(s,x)\big)\, ds \quad\hbox{ for a.e. in} \; [0,T]\times \Omega.
 \end{equation}
 In particular, $v$ is an $L^r$ mild solution to the Navier-Stokes equations \eqref{eq:NSE} discussed in \cite{GigaM}. 
\end{thm}
\begin{proof}
Let us define  $n=n_{\tau}:=\lfloor \frac{t}{\tau}\rfloor$ for a given time $t\in [0,T]$. Using the approximate Duhamel formula in Lemma \ref{lem:approximate_duhamel}, 
we see that for any smooth divergence free test function $f$ whose is contained in $\Omega$ we have 
\[
(v_{n+1,\tau},f)=(v_0, f_{n+1})+\tau\sum_{k=0}^n \Big(v_{k,\tau}, \frac{f_{n+1-k}\circ Z_{k+1,\tau}-f_{n+1-k}}{\tau}\Big).
\]
We can then write
\[
(v_{n+1,\tau},f)=(v_0, f_{n+1})+\tau\sum_{k=0}^n \Big(v_{k,\tau}\otimes v_{k,\tau}, Df_{n+1-k} \Big)+\epsilon_{\tau}
\]
where 
\[
\epsilon_{\tau}:= \tau\sum_{k=0}^n  \Big(v_{k,\tau}, \frac{f_{n+1-k}\circ Z_{k+1,\tau}-f_{n+1-k}}{\tau} -Df_{n+1-k} v_{k,\tau}).
\]
Next, observe 
\[
|\epsilon_{\tau}|\leq  \tau\sum_{k=0}^n  \int_{\Omega}|v_{k,\tau}(x)||Df_{n+1-k}(x)|| v_{k,\tau}(x)-\tilde{v}_{k+1,\tau}|\, dx+
\]
\[
\tau^2\sum_{k=0}^n\int_{\Omega} \int_0^1\int_0^t |D^2 f_{n+1-k}\big(sZ_{k+1,\tau}(x)+(1-s)x\big)||v_{k,\tau}(x)||\tilde{v}_{k+1,\tau}|^2\, ds\, dt\, dx. 
\]
Thanks to the Sobolev Embedding Theorem we have 
\begin{multline}\label{eq:epsilon_tau}
|\epsilon_{\tau}|\lesssim_{d,r}  \norm{D^2f_{n+1-k}}_{L^{r}(\Omega)}\norm{v_{\tau}}_{L^{2}([0,T]\times\Omega)} \tau\sum_{k=0}^n\norm{v_{k,\tau}-\tilde{v}_{k+1,\tau}}_{L^2(\Omega)}^2+\\
\tau\norm{D^3f_{n+1-k}}_{L^{r}(\Omega)}\norm{v_{\tau}}_{L^{\infty}([0,T]\times\Omega)}\norm{\tilde{v}_{\tau}}_{L^2([0,T]\times\Omega)}^2. 
\end{multline} 
Using  Proposition \ref{prop:scheme_ekeland_estimates} we have 
\[
\norm{v_{k,\tau}-\tilde{v}_{k+1,\tau}}_{L^2(\Omega)}\leq (\mt)^{1/2}\norm{Dv_{n,\tau}}_{L^2(\Omega)}+\mu^{-\frac{d+r}{2r}} \tau^{1-\frac{d+r}{2r}}\norm{v_{n,\tau}}_{L^r(\Omega)}^2.
\]
Combining the above bound with  \eqref{EDI}, we can conclude that 
\[
\tau\sum_{k=0}^n\norm{v_{k,\tau}-\tilde{v}_{k+1,\tau}}_{L^2(\Omega)}\leq (\tau T)^{1/2}\norm{v_0}_{L^2(\Omega)}+\mu^{-\frac{d+r}{2r}} \tau^{1-\frac{d+r}{2r}} T\norm{v_{\tau}}_{L^{\infty}([0,T];L^r(\Omega))}^2.
\]
Recall that $v_{k,\tau}:=e^{\mt A} w_{k,\tau}$ in (\ref{eq:ns_scheme_step_2}) .  Thus Sobolev inequalities and Lemma \ref{lem:para_est_0} yields 
\[
\norm{v_{\tau}}_{L^{\infty}([0,T]\times\Omega)}\lesssim_{d,r} \norm{D v_{\tau}}_{L^{\infty}([0,T];L^{r}(\Omega))}\lesssim_{d,r} (\mt)^{-\frac{d+r}{2r}}\max_{0\leq k\leq n} \norm{w_{k,\tau}}_{L^2(\Omega)}.
\]
Recalling our argument in the proof of  \eqref{EDI}, we have 
\[
\max_{0\leq k\leq n} \norm{w_{k,\tau}}_{L^2(\Omega)}\leq \max_{0\leq k\leq n} \norm{v_{k,\tau}}_{L^2(\Omega)}\leq \norm{v_0}_{L^2(\Omega)}.
\]
Plugging the above estimates into formula (\ref{eq:epsilon_tau}), we can now conclude that 
$\lim_{\tau\to 0} |\epsilon_{\tau}|=0$
for all divergence free $f\in W^{3,r}_0(\Omega)$.

By Proposition \ref{prop:precompactness}, along a subsequence $v_{\tau}$ converges to $v \in L^2([0,T]\times\Omega)$.  It is now clear that for any divergence free $f\in L^{\infty}([0,T];W^{3,r}_0(\Omega))$ we have
\[
\int_0^T\int_{\Omega} v(t,x)f(t,x)\, dt\, dx   =\lim_{j\to\infty} \int_0^T\int_{\Omega} v_{\tau_j}(t,x) f(t,x)\, dx\, dt
\]
\[
=\lim_{j\to\infty} \int_0^T v_0(x) e^{-\mu \tau_j\lfloor t/\tau_j\rfloor \mathcal{A}} f(t,x)\, dt+\int_0^T \int_0^t \int_{\Omega} v_{\tau_j}(s,x)\otimes v_{\tau_j}(s,x) D e^{-\mu \tau_j\lfloor \frac{t-s}{\tau_j}\rfloor \mathcal{A}}f(t,x) \, ds\, dx\, dt
\]
\[
=\int_0^T v_0(x) e^{-\mu t \mathcal{A}} f(t,x)\, dt+\int_0^T \int_0^t \int_{\Omega} v(s,x)\otimes v(s,x) D e^{-\mu(t-s) \mathcal{A}}f(t,x) \, ds\, dx\, dt.
\]

Since $v\in L^2([0,T];H^1(\Omega))\cap L^{\infty}([0,T];L^r(\Omega))$ we have
\[
\nabla \cdot (v\otimes v)\in L^1([0,T]\times\Omega).
\]
Hence, we can conclude that, for a.e. $(t,x)\in [0,T]\times \Omega$,
\[ 
v(t,x)= e^{-\mu t\mathcal{A}}v_0 (x)- \int_0^t e^{-\mu (t-s)\mathcal{A}} \nabla \cdot \big( v(s,x)\otimes v(s,x)\big)\, ds.
\]

It is a straightforward consequence of the estimates in Lemma \ref{lem:para_est_0} that any $L^{\infty}([0,T];L^r(\Omega))$ solution to (\ref{eq:duhamel_solution}) with $r>d$ must be unique (see Theorem~\ref{thm:F1}).  Thus,  the full sequence $\{v_{\tau}\}_{\tau>0}$ converges to $v$ as $\tau\to 0$.

\end{proof}

\begin{thm}\label{thm:conv_2}
 For $T^*$ and $v$  as in Theorem~\ref{thm:conv_2}, there is a unique $X:[0,T]\times\Omega\to\Omega$ such that
 \begin{equation}\label{ode_X}
 X(t,x)=\id+\int_0^t v(s,X(s,x))\, ds, \quad \det(DX(t,x))=1 \hbox{ a.e.} ,
 \end{equation}
 and 
 \[
 \lim_{\tau\to 0} \norm{X_{\tau}-X}_{L^1([0,T]\times\Omega)}=0.
 \]
\end{thm}
\begin{proof}
Note that  any $v$, solution to equation (\ref{eq:duhamel_solution}) in $L^{\infty}([0,T];L^r(\Omega))\cap L^2([0,T];H^1(\Omega))$ is also in $L^1([0,T];W^{1,\infty}(\Omega))$ (see Theorem~\ref{thm:F1}).  Thus $X$ is well-defined by \eqref{ode_X}. 

It remains to show that $X_{\tau}$ converges to $X$. Recall that from (\ref{eq:ns_scheme_step_1}-\ref{eq:ns_scheme_step_3})
\[
X_{n+1,\tau}(x)=\id+\tau \sum_{k=0}^n \tilde{v}_{k+1,\tau}(X_{k,\tau}(x)).
\]
and by definition
\[
X_{\tau}(t,x)=\id+\int_0^t v_{\tau}(s, X_{\tau}(s,x))\, ds+\tau\sum_{k=0}^n \tilde{v}_{k+1,\tau}(X_{k,\tau}(x))-v_{k+1,\tau}(X_{k+1,\tau}(x)).
\]
Therefore, for any $t\in [0,T]$, we have the estimate
\[
\norm{X_{\tau}(t,\cdot)-X(t,\cdot)}_{L^1(\Omega)}\leq \norm{v-v_{\tau}}_{L^1([0,t]\times\Omega)}+\int_0^t \norm{Dv(s,\cdot)}_{L^{\infty}(\Omega)}\norm{X_{\tau}(s,\cdot)-X(s,\cdot)}_{L^1(\Omega)}\, ds+
\]
\[
\tau\sum_{k=0}^n \norm{v_{k+1,\tau}\circ Z_{k+1,\tau}-v_{k+1,\tau}}_{L^1(\Omega)}+\norm{v_{k,\tau}-\tilde{v}_{k+1,\tau}}_{L^1(\Omega)},
\]
where we have used the fact that $X_{k,\tau}$ is measure preserving for all $k$ and $X_{k+1,\tau}=Z_{k+1,\tau}\circ X_{k,\tau}$. 
Applying Remark \ref{rem:w1p_lip} and Proposition \ref{prop:scheme_ekeland_estimates}, we have 
\[
\tau\sum_{k=0}^n \norm{v_{k+1,\tau}\circ Z_{k+1,\tau}-v_{k+1,\tau}}_{L^1(\Omega)}+\norm{v_{k,\tau}-\tilde{v}_{k+1,\tau}}_{L^1(\Omega)}\leq 
\]
\[
\tau\norm{Dv_{\tau}}_{L^2([0,T]\times\Omega)}\norm{\tilde{v}_{\tau}}_{L^2([0,T]\times\Omega)}+ (t\mt)^{1/2}\norm{Dv_{\tau}}_{L^2([0,T]\times\Omega)}+\mu^{-\frac{d+r}{2r}} \tau^{1-\frac{d+r}{2r}}\norm{v_{\tau}}_{L^2([0,t];L^r(\Omega))}^2.
\]
Now we can use Gronwall's inequality to conclude that
\[
\norm{X_{\tau}(t,\cdot)-X(t,\cdot)}_{L^1(\Omega)}\leq \delta_{\tau}\exp(\norm{Dv}_{L^1([0,t];L^{\infty}(\Omega))}),\,\, \hbox { with } \lim_{\tau\to 0}\delta_{\tau}=0.
\]

\end{proof}

\begin{remark}\label{eq:last-remark}  Theorem~\ref{thm:F1} yields that $v\in L^1\big((0,T^*);W^{1,\infty}(\Omega)\big)$ with $v(\cdot,0)\in L^r(\Omega)$, which is enough to conclude that $X$ is unique and $X(t, \cdot)$ is one-to-one of $\bar \Omega$ onto itself. From here, one would use the standard theory for \eqref{eq:NSE} to improve the regularity properties of $v$  to $v \in L^1\big((0,T^*); C^{1,\alpha}(\Omega)\big)$ for some $\alpha>0$. Since $X$ satisfies \eqref{ode_X} with $v=0$ on $\partial\Omega$, we can conclude that $X \in L^{\infty}\big((0,T^*);\mathcal{D}\textup{iff}_{\id}(\Omega)\big)$. 
\end{remark}

\appendix

\section{Inequalities} 
The following Lemma is a classical result from the theory of maximal functions which can be found in \cite{DeVoreS}.  
\begin{lemma}\label{lem:w1p_lip}
For any $l\in [1,\infty]$ and $f\in W^{1,p}(\RR^d),$ 
\[
|f(x)-f(y)|\leq \Big(M (\nabla f)(x)+M (\nabla f)(y)\Big)|x-y|, \quad \text{for a.e.} \quad x,y\in\RR^d.
\]
Here, $M$ denotes the Hardy-Littlewood maximal function. Therefore, if $1<l\leq \infty$, there exists a constant $C\equiv C_l(d)$ such that $\|M (\nabla f)\|_{L^l(\Omega)} \leq C \|\nabla f\|_{L^l(\Omega)}.$
\end{lemma}
\begin{remark}\label{rem:w1p_lip} Assume that $\Omega$ is an open bounded set of class $C^3$. Let $l\in (1,\infty]$ and let $f\in W^{1,p}_0(\Omega).$ Denote as $\tilde f \in W^{1,p}(\RR^d)$ the extension of $f$ which is identically null outside $\Omega$. Let $g$ be the restriction of $M (\nabla \tilde f)$ to $\Omega.$   We have 
\[
 \|g\|_{L^l(\Omega)} \leq \|M (\nabla \tilde f)\|_{L^l(\RR^d)} \leq  \tilde C  \|\nabla \tilde f\|_{L^l(\RR^d)} = \tilde C  \|\nabla f\|_{L^l(\Omega)}
\]
\begin{enumerate} 
\item[(i)] We have 
\[
|f(x)-f(y)|\leq (g(x)+g(y))|x-y|, \quad \text{for a.e.} \quad x,y\in\Omega.
\]
\item[(ii)] Consequently, if $r>2$ and $Z: \Omega \rightarrow \Omega$ preserves Lebesgue measure then 
\[
\|f(Z)-f\|_{L^{r\over r-1}(\Omega)} \leq 2C \|Z-\id\|_{L^{r}(\Omega)}\|Df\|_{L^{r \over r-2}(\Omega)}
\]
\end{enumerate}
\end{remark}

\begin{lemma}\label{lem:discrete_nonlinear_gronwall}
Suppose that $g:[0,\infty)\to\RR$ is an increasing function and $\{a_k\}_{k\geq 0}$, $\{b_k\}_{k\geq 0}$, $\{c_k\}_{k\geq 0}$, and $\{\beta_{n,k}\}_{n,k\geq 0}$ are nonnegative sequences  such that
\[
a_{n+1}\leq c_{n+1} a_0+\sum_{k=0}^{n} \beta_{k,n}\, g(a_k)
\]
\[
b_{n+1}= b_0+\sum_{k=0}^{n} \beta_{k,n}\, g(b_k).
\]
If $c=\sup_{k\geq 0} c_k$ is finite and $\max(c,1) a_0\leq b_0$ then $a_n\leq b_n$ for all $n\geq 0$. 
 
\end{lemma}
\begin{proof}
By assumption $a_0\leq b_0$, hence it suffices to show that $a_k\leq b_k$ for all $k\leq n$ implies that $a_{n+1}\leq b_{n+1}$.  Using the formulas in the assumption of the Lemma, we have
\[
a_{n+1}-b_{n+1}\leq c_{n+1}a_0-b_0+\sum_{k=0}^{n} \beta_{k,n} \big(g(a_k)-g(b_k)\big).
\]
The result follows from the induction hypothesis, the monotonicity of $g$, and the nonnegativity of each sequence.  
\end{proof}

%
%
\section{Flows on $\mathcal D{\rm iff}_{\id}(\Omega)$}\label{sec:standard}

Here we provide a completely standard lemma guaranteeing the existence of certain paths in $\mathcal{D}\textup{iff}_{\id}(\Omega)$. 
\begin{lemma}\label{lem:aug20.4} For some $r\in (d,\infty)$ suppose that $u \in W^{2, r}(\Omega, \mathbb R^d)$ is a divergence free vector field, and let $\mathcal{M}=Z\in W^{2,r}_{\id}(\Omega)\cap \mathcal{D}\textup{iff}_{\id}(\Omega)$.  If $Z\in \mathcal{M}$,  there exists a flow $\gamma:\RR\to \mathcal{M}$ such that 
\[
\gamma(0)=Z, \quad \gamma'(0)=u\circ Z.
\]

\end{lemma}
\begin{proof} 
Since $u$ is divergence free and Lipschitz, there exists a solution $h$ to the ODE
\[
h(0)=\id, \quad h'(0)=u.
\]
Gronwall's Lemma implies that
\[
\norm{Dh(t)}_{L^{\infty}(\Omega)}\leq \exp(t \norm{Du}_{L^{\infty}(\Omega)}).
\]
Since $u\in W^{2,r}_0(\Omega)$ is divergence free, it then follows that $h|_{\partial\Omega}(t)=\id$ and $\det(Dh(t))\equiv 1$.  Therefore $h\in \mathcal{D}\textup{iff}_{\id}(\Omega)$. 
We can then estimate
\[
\norm{D^2 h'(t)}_{L^r(\Omega)}\leq \norm{Du(h(t))}_{L^{\infty}(\Omega)}\norm{D^2h(t)}_{L^r(\Omega)}+\norm{D^2u(h(t))}_{L^r(\Omega)}\norm{Dh(t)}_{L^{\infty}(\Omega)}^2. 
\]
Using the fact that $h$ is measure preserving and the previous bound,
\[
\norm{D^2 h'(t)}_{L^r(\Omega)}\leq \norm{Du}_{L^{\infty}(\Omega)}\norm{D^2h(t)}_{L^r(\Omega)}+\norm{D^2u}_{L^r(\Omega)} \exp(2t \norm{Du}_{L^{\infty}(\Omega)}).
\]
Therefore, by Gronwall's Lemma, $h(t)\in W^{2,r}(\Omega)$ for all finite times and so $h:\RR\to\mathcal{M}$.    Now let us set $\gamma(t)=h(t)\circ Z$.  Since $\mathcal{M}$ is closed with respect to composition, we see that $\gamma:\RR\to\mathcal{M}$.  Finally, it is clear that $\gamma(0)=Z$ and $\gamma'(0)=u\circ Z$. 
  \end{proof}

 \section{Stokes operator estimates}\label{section:Stokes}
 
 We will use the following estimate on the Stokes resolvent problem.

\begin{lemma}\label{lem:para_est_0}
For any $q\in [1,\infty)$, $p\in [1,\infty]$, $t>0$ and any $f \in K$, we have
\begin{equation}
\norm{\nabla e^{-t \mathcal{A}}f}_{L^{p}(\Omega)}  \leq C_{q,d} t^{-\sigma}\norm{f}_{L^q(\Omega)} , \hbox{ where } \sigma= \frac{d}{2}\max(q^{-1}-p^{-1},0) + \frac{1}{2}.
\end{equation}

\end{lemma}

\begin{proof}
We have the following estimates for any $1\leq l\leq n \leq \infty$. For $u$ sufficiently smooth, 
\begin{equation}\label{first}
\|D^2u\|_l \leq C \|\mathcal{A}u\|_l; \qquad \hbox{ (\cite{RTemam}, Chapter 1.2, Prop 2.2)} 
\end{equation}
\begin{equation}\label{third}
\|\mathcal{A}e^{-t \mathcal{A}} f\|_l\leq {C_l \over t}\|f\|_l, \quad \forall f \in K;\qquad \hbox{(\cite{GigaM}, Prop 1.2)}
\end{equation}
\begin{equation}\label{second} 
\|e^{-t \mathcal{A}} f\|_{L^{n}(\Omega)} \leq {C\over t^k}\|f\|_{L^{l}(\Omega)}, \quad \forall f \in K, \qquad \hbox{where} \quad k= \Big(l^{-1}- n^{-1}) \frac{d}{2}.  \qquad \hbox{(\cite{Giga86}, (A))}
\end{equation}
To prove the lemma, we choose $\alpha$ in the GNS inequality 
$$
\norm{Du}_{L^p(\Omega)}\leq C\norm{D^2u}^{\alpha}_{L^q(\Omega)} \norm{u}_{L^{q}(\Omega)}^{1-\alpha},
$$
so that
\begin{equation}\label{GNS}
\frac{1}{p} \geq \frac{1}{d} + \Big(\frac{1}{q}-{2\over d}\Big)\alpha + {(1-\alpha) \over q} \quad  \text{and} \quad \alpha \in [1/2, 1].
\end{equation}
Hence, the smallest choice we can make for $\alpha$ is $\alpha=\frac{1}{2}+\frac{d}{2}\max(q^{-1}-p^{-1},0)$.

Let $u := e^{tA}f$. From \eqref{first}-\eqref{third} we have
$$
\|D^2u\|_q \leq \|Au\|_q \leq  {1 \over t} \|f\|_q,
$$
and from \eqref{second}
$$
\|u\|_{q} \leq C \|f\|_q .
$$
 Hence  we have
 $$
 \|Du\|_\gamma \leq  {C \over t^{\alpha}} \|f\|_q.
 $$
Using our choice of $\alpha$ we conclude.
\end{proof}

 \section{Navier-Stokes basic estimates}

The following theorem is expected to be classical, however we were not able to locate an explicit reference. Thus we provide a proof for the completeness.

\begin{thm} \label{thm:F1} For any $v_0\in L^r(\Omega)$ which is divergence free, there is at most one  $v\in L^{\infty}([0,T];L^r(\Omega))$ that satisfies the Duhamel's formula
\begin{equation}\label{eq:true_duhamel}
v(t)=e^{-\mu t \mathcal A} v_0-\int_0^t e^{-\mu(t-s)  \mathcal A} \PP D^T(v(s)\otimes v(s))\, ds.
\end{equation}
Moreover $ t^{d+r \over 2r} v(x,t) \in L^{\infty}((0,T];W^{1,\infty}(\Omega))$. In particular $v\in L^1((0,T]; W^{1,\infty}(\Omega))$. 
\end{thm}
\begin{proof}
Let us denote $\sigma_{\alpha}(t):=t^{\alpha}$ and $A_v:= \norm{v}_{L^{\infty}([0,T];L^r(\Omega))}<\infty$. We begin by showing that $\sigma_{\frac{d}{2r}} v\in L^{\infty}((0,T]\times\Omega)$. Note that if $\phi$ is a smooth divergence free vector field that vanishes on the boundary of $\Omega$ then 
\[
\Big(e^{-\mu(t-s)  \mathcal A}\PP D^T(v(s)\otimes v(s)), \phi \Big)=- \Big(v(s)\otimes v(s),D\big(e^{-\mu(t-s)  \mathcal A} \phi\big) \Big).
\]
We use Lemma \ref{lem:para_est_0} to conclude that 
\[
\Big| \Big(e^{-\mu(t-s)  \mathcal A}\PP D^T(v(s)\otimes v(s)), \phi \Big)\Big| \lesssim_r \norm{v(s)\otimes v(s)}_{L^r(\Omega)} \big(\mu(t-s))^{-{d+r\over 2r}}\norm{\phi}_{L^1(\Omega)}.
\]
Using the Duhamel formula (\ref{eq:true_duhamel}) and \eqref{second}, we can conclude that 
$$
\begin{array}{lll}
\norm{v(t) \,\sigma_{ \frac{d}{2r}}(t))}_{L^{\infty}(\Omega)}&\lesssim& \norm{v_0}_{L^r(\Omega)}+t^{\frac{d}{2r}}\int_{0}^t (t-s)^{-\frac{d+r}{2r}}\norm{v(s) \otimes v(s)}_{L^{r}(\Omega)} ds\\ \\
&\lesssim & \norm{v_0}_{L^r(\Omega)}+t^{\frac{d}{2r}}\int_{0}^t (t-s)^{-\frac{d+r}{2r}}\norm{v(s)}_{L^{\infty}(\Omega)}\norm{v(s)}_{L^r(\Omega)}\, ds. \\ \\
&\lesssim& \norm{v_0}_{L^r(\Omega)}+t^{\frac{d}{2r}}A_v\int_{0}^t s^{-\frac{d}{2r}} (t-s)^{-\frac{d+r}{2r}}\norm{v(s) \,\sigma_{\frac{d}{2r}}(s)}_{L^{\infty}(\Omega)}\, ds.
\end{array}
$$
Since 
$$s^{-\frac{d}{2r}} (t-s)^{-\frac{d+r}{2r}} \lesssim t^{-\frac{d+r}{2r}} s^{-\frac{d}{2r}}\chi_{0\leq s\leq t/2} +t^{-\frac{d}{2r}} (t-s)^{-\frac{d+r}{2r}}\chi_{t/2 < s\leq 1}
$$
we have 
\begin{align*}
\norm{v(t)\,\sigma_{\frac{d}{2r}}(t)}_{L^{\infty}(\Omega)}\lesssim & \norm{v_0}_{L^r(\Omega)}\\
+ &A_v\bigg( t^{-\frac{1}{2}}\int_{0}^{t/2} s^{-\frac{d}{2r}} \norm{v(s)\, \sigma_{\frac{d}{2r}}(s)}_{L^{\infty}(\Omega)} ds+\int_{t/2}^t (t-s)^{-\frac{d+r}{2r}}\norm{v(s)\,\sigma_{\frac{d}{2r}}(s)}_{L^{\infty}(\Omega)} ds   \bigg).
\end{align*}
Choose some $\beta<\infty$ such that $\frac{\beta(d+r)}{2r(\beta-1)}<1$, this is possible since $r>d$.  Using H\"{o}lder's inequality, we obtain the estimate
\begin{equation}\label{eq:duhamel_lp_control_3}
\norm{v(t)\, \sigma_{\frac{d}{2r}}(t)}_{L^{\infty}(\Omega)}\lesssim \norm{v_0}_{L^r(\Omega)}+
 t^{\frac{\beta-1}{\beta}-\frac{d+r}{2r}}A_v h(t)^{\beta^{-1}},
\end{equation}
where
\[
h(t):=\int_0^t  \norm{v(s)\, \sigma_{d \over 2r}(s) }^\beta_{L^\infty(\Omega)}. 
\]
Raising both sides to the power $\beta$, and integrating over $[0,t_0]$ with $t_0<T$, we have 
\[
h(t_0)\lesssim \norm{v_0}_{L^r(\Omega)}^{\beta}+
 A_v^{\beta}\int_0^{t_0} t^{\beta-1-\frac{\beta(d+r)}{2r}} h(t)\, dt.
\]
Thus, Gronwall's inequality implies that $h$ stays finite in $[0,T]$.  Plugging this result back into (\ref{eq:duhamel_lp_control_3}) and recalling that $ \frac{\beta-1}{\beta}-\frac{d+r}{2r}>0$, we conclude that $\sigma_{\frac{d}{2r}}v\in L^{\infty}((0,T]\times\Omega)$ as desired. 


Next we will use the above $L^{\infty}$ bound to show that $\sigma_{\frac{1}{2}}v\in L^{\infty}((0,T];W^{1,r}_0(\Omega))$.
Once again, using the Duhamel formula (\ref{eq:true_duhamel}) and Lemma \ref{lem:para_est_0}, we can conclude that 
$$
\begin{array}{lll}
\norm{\sigma_{\frac{1}{2}}(t)Dv(t)}_{L^{r}(\Omega)}&\lesssim&
 \norm{v_0}_{L^r(\Omega)}+t^{\frac{1}{2}}\int_{0}^t (t-s)^{-\frac{1}{2}}\norm{v(s)}_{L^{\infty}(\Omega)}\norm{Dv(s)}_{L^{r}(\Omega)}\, ds.\\ \\
&\lesssim&
 \norm{v_0}_{L^r(\Omega)}+t^{\frac{1}{2}}\norm{v\, \sigma_{\frac{d}{2r}}}_{L^{\infty}((0,T]\times\Omega)}\int_{0}^t s^{-\frac{d+r}{2r}} (t-s)^{-\frac{1}{2}}\norm{\sigma_{\frac{1}{2}}(s)Dv(s)}_{L^{r}(\Omega)}\, ds.
\end{array}
$$
We can argue as above with the same choice of $\beta$ to conclude that $\sigma_{\frac{1}{2}}v\in L^{\infty}((0,T];W^{1,r}_0(\Omega))$.

Now we show  $\sigma_{\frac{d+r}{2r}}v\in L^{\infty}((0,T];W^{1,\infty}_0(\Omega))$.   Again by the Duhamel formula (\ref{eq:true_duhamel}) and Lemma \ref{lem:para_est_0}, we have the estimate 
\[
\norm{\sigma_{\frac{d+r}{2r}}(t)Dv(t)}_{L^{\infty}(\Omega)}\lesssim\\
 \norm{v_0}_{L^r(\Omega)}+t^{\frac{d+r}{2r}}\int_{0}^t (t-s)^{-\frac{d+r}{2r}}\norm{v(s)}_{L^{\infty}(\Omega)}\norm{Dv(s)}_{L^{r}(\Omega)}\, ds.
\]
Using our previous work, we get the bound
\begin{align*}
\norm{\sigma_{\frac{d+r}{2r}}(t)Dv(t)}_{L^{\infty}(\Omega)}\lesssim & \norm{v_0}_{L^r(\Omega)}\\
 +& t^{\frac{d+r}{2r}}\norm{v\, \sigma_{\frac{d}{2r}}}_{L^{\infty}((0,T]\times\Omega)}\norm{\sigma_{\frac{1}{2}}Dv}_{L^{\infty}((0,T];L^r(\Omega))} \int_{0}^t s^{-\frac{d+r}{2r}} (t-s)^{-\frac{d+r}{2r}}\, ds.
\end{align*}
The expression
\[
t^{\frac{d+r}{2r}}\int_{0}^t s^{-\frac{d+r}{2r}} (t-s)^{-\frac{d+r}{2r}}\, ds
\]
is uniformly bounded for all $t$, therefore we obtain $v\, \sigma_{\frac{d+r}{2r}}\in L^{\infty}((0,T];W^{1,\infty}_0(\Omega))$ as desired. 

Finally, we can prove the uniqueness of $v$.  Suppose that $\tilde{v}\in L^{\infty}([0,T];L^r(\Omega))$ also satisfies equation (\ref{eq:true_duhamel}).  Lemma \ref{lem:para_est_0} then yields 
\[
\norm{v(t)-\tilde{v}(t)}_{L^1(\Omega)}\lesssim \int_0^t (t-s)^{-\frac{1}{2}}\norm{v(s)-\tilde{v}(s)}_{L^1(\Omega)}\big(\norm{v(s)}_{L^{\infty}(\Omega)}+\norm{\tilde{v}(s)}_{L^{\infty}(\Omega)}\big)\, ds.
\]
Using the $L^{\infty}$ bounds we deduced above, we have 
\[
\norm{v(t)-\tilde{v}(t)}_{L^1(\Omega)}\lesssim \int_0^t s^{-\frac{d}{2r}}(t-s)^{-\frac{1}{2}}\norm{v(s)-\tilde{v}(s)}_{L^1(\Omega)} ds\lesssim t^{1/6}\norm{v-\tilde{v}}_{L^3([0,t];L^1(\Omega))}.
\]
Hence,  $\dot H(t) \lesssim t^{1/2} H(t)$ if we set $$H(t):= \int_0^t \norm{v(s)-\tilde{v}(s)}_{L^1(\Omega)}^3ds.$$ Thus, $H(t) \lesssim H(0) e^{{2\over 3}t^{3 \over 2}}.$ This yields $H \equiv 0$, which means that $v=\tilde{v}$ for $0\leq t\leq T$. 
\end{proof}

\providecommand{\href}[2]{#2}

\end{document}